\documentclass[11pt]{article}
\usepackage{natbib}
\usepackage{amsmath, amsfonts, amsthm, amssymb, mathtools}
\usepackage{authblk, color, bm, bbm, graphicx, epstopdf}
\usepackage[labelfont = bf, labelsep = period]{caption} 
\usepackage{subcaption, xspace, enumitem,adjustbox}
\usepackage{booktabs, tabularx, multirow, rotating, array,lscape,hyperref,makecell}
\usepackage[bottom,hang]{footmisc}
\usepackage[flushleft]{threeparttable}
\usepackage{accents}

\makeatletter
\newcommand{\mylabel}[2]{#2\def\@currentlabel{#2}\label{#1}}
\makeatother
\usepackage{subcaption,longtable}
\usepackage{tikz,pgfplots}
\usetikzlibrary{positioning,arrows.meta}
\usetikzlibrary{shapes.geometric, arrows,calc}
\tikzstyle{decision} = [diamond, minimum width=0.5cm, minimum height=0.5cm, text centered, draw=black, very thick, text width=4cm, aspect=2]
 \tikzstyle{algorithmStep} = [rectangle, minimum width=3cm, minimum height=1cm,text centered, draw=black, very thick, text width=5cm]
 \tikzstyle{algorithmStep2} = [rectangle, minimum width=3cm, minimum height=1cm,text centered, draw=black, very thick, text width=2cm]
 \tikzstyle{arrow} = [thick,->,>=stealth]
 \tikzstyle{input} = [trapezium, trapezium left angle = 65, trapezium right angle = 115, minimum width=3cm, minimum height=1cm,text centered, draw=black, very thick, text width=2cm]
\tikzstyle{decision2} = [diamond, minimum width=0.25cm, minimum height=0.25cm, text centered, draw=black, very thick, text width=2cm, aspect=2]
\usepackage{pgfplots}
\usetikzlibrary{intersections}
\usetikzlibrary{shapes, arrows, chains}
\usetikzlibrary{arrows.meta}

\newcommand{\x}{\boldsymbol{x}}
\newcommand{\twodbounds}[4]{\makecell[tl]{$x_1^L$ : #1 \\ $x_1^U$ : #2 \\ $x_2^L$ : #3 \\ $x_2^U$ : #4}}
\newcommand{\twodboundsundef}[4]{\makecell[tl]{$x_1^{L*}$ : #1 \\ $x_1^{U*}$ : #2 \\ $x_2^{L*}$ : #3 \\ $x_2^{U*}$ : #4}}
\newcommand{\threedbounds}[6]{\makecell[tl]{$x_1^L$ : #1 \\ $x_1^U$ : #2 \\ $x_2^L$ : #3 \\ $x_2^U$ : #4 \\ $x_3^L$ : #5 \\ $x_3^U$ : #6}}

\DeclareMathOperator*{\argmin}{\arg\!\min}

\newtheorem{innercustomlem}{Lemma}
\newenvironment{lem}[1]
  {\renewcommand\theinnercustomlem{#1}\innercustomlem}
  {\endinnercustomlem}

\newtheorem{obs}{Observation}

\newtheorem{rem}{Remark}

\newtheorem{property}{Property}
\newtheorem{assump}{Assumption}

\usepackage{fullpage}[2cm]
\title{Constructing Tight Quadratic Relaxations for Global Optimization: I. Outer-Approximating Twice-Differentiable Convex Functions}

\author{William~R.~Strahl$^{1,2}$}
\author{Arvind~U.~Raghunathan$^3$}
\author{Nikolaos~V.~Sahinidis$^{2,4}$}
\author{Chrysanthos!E.~Gounaris$^{1,2}$\thanks{Corresponding author: gounaris@cmu.edu}}

\affil{\small $^1$Department of Chemical Engineering, Carnegie Mellon University, Pittsburgh, PA, 15213, USA \\ 
			  $^2$Center for Advanced Process Decision-making, Carnegie Mellon University, Pittsburgh, PA, 15213, USA \\ 
			  $^3$Mitsubishi Electric Research Labs, Cambridge, MA, 02139, USA \\ 
			  $^4$H.~Milton Stewart School of Industrial~\& Systems Engineering and School of Chemical~\& Biomolecular Engineering, Georgia Institute of Technology, Atlanta, GA, 30332, USA}

\date{}

\begin{document}

\maketitle

\begin{abstract}
When computing bounds, spatial branch-and-bound algorithms often linearly outer approximate convex relaxations for non-convex expressions in order to capitalize on the efficiency and robustness of linear programming solvers. Considering that linear outer approximations sacrifice accuracy when approximating highly nonlinear functions and recognizing the recent advancements in the efficiency and robustness of available methods to solve optimization problems with quadratic objectives and constraints, we contemplate here the construction of quadratic outer approximations of twice-differentiable convex functions for use in deterministic global optimization.
To this end, we present a novel cutting-plane algorithm that determines the tightest scaling parameter, $\alpha$, in the second-order Taylor series approximation quadratic underestimator proposed by~\citet{su2018improved}. We use a representative set of convex functions extracted from optimization benchmark libraries to showcase--qualitatively and quantitatively--the tightness of the constructed quadratic underestimators and to demonstrate the overall computational efficiency of our algorithm. Furthermore, we extend our construction procedure to generate even tighter quadratic underestimators by allowing overestimation in infeasible polyhedral regions of optimization problems, as informed by the latter's linear constraints.

\noindent \textbf{Keywords:} deterministic global optimization, outer approximation, quadratic underestimation, cutting-plane algorithm
\end{abstract}

\section{Introduction}\label{sec:introduction}
As evidenced by the vast diversity of application areas present in the literature, deterministic global optimization (DGO) algorithms are utilized to solve problems to guaranteed global optimality in many different fields, including biomedical and biological engineering, computational chemistry, computational geometry, finance, process networks, and others. A brief, yet diverse subset of specific applications where deterministic DGO algorithms have been applied includes bioreactors~\citep{misener2014global}, protein design~\citep{klepeis2004design}, phase equilibrium~\citep{pereira2010duality}, membrane-based separation systems~\citep{gounaris2013estimation}, investment portfolio selection~\citep{rios2010portfolio}, and crude oil scheduling~\citep{mouret2011new}. Recent applications include optimizing heat integration in flowsheets~\citep{fahr2022simultaneous}, parameter estimation problems in the chemical industry~\citep{sass2023towards}, and cost minimization of water distribution networks~\citep{cassiolato2023minlp}, to name but a few. Indeed, advancing DGO techniques produces enhanced algorithms that can be utilized in a plethora of problems encountered across a vast array of research fields.

Many DGO algorithms extensively utilize supporting hyperplanes (i.e., first-order Taylor series approximations) for nonlinear convex functions, which have guarantees from convexity theory to underestimate convex functions over the entire domain (e.g.,~\citep{tawarmalani2005polyhedral}). Most notably, spatial branch-and-bound algorithms, which repeatedly construct and refine convex relaxations of the typically non-convex feasible space, often opt for linear outer-approximations of those relaxations at the interest of solving them faster and more reliably. 
Whereas they supply valid underestimators necessary for rigorous guarantees of optimality and are very inexpensive to construct, linear outer-approximators can poorly approximate highly nonlinear functions. More broadly, directly embedding the nonlinear convex functions into any intermediate computation of a DGO algorithm immediately elevates the difficulty of the problem to solving general convex optimization problems, for which efficient methods exist, but are computationally less efficient and/or less robust compared to linear programming (LP) algorithms.
Moreover, recent advances in the algorithms solving problems with quadratic objectives and/or constraints~\citep{hansmittlemann} suggests an intermediate option: using a nonlinear quadratic function as an outer approximator to more accurately approximate nonlinear functions, while at the same time maintaining a lower level of difficulty than general nonlinear convex programming. This paradigm has been applied, for example, in a decomposition algorithm for convex Mixed Integer Nonlinear Programming (MINLP) problems, where an Outer Approximation (OA) scheme generates quadratic cuts defining a Mixed Integer Quadratically Constrained Quadratic Program (MIQCQP) relaxed master problem~\citep{su2018improved}. Furthermore, any algorithm or technique that linearly outer approximates convex relaxations could instead consider quadratic outer approximations, which provide an intermediate alternative within the complexity gap between nonlinear convex programming and linear programming, allowing for enhanced accuracy of relaxation solutions at a greater computational cost. In this context, another possible application lies in optimization-based bounds tightening (OBBT)~\citep{puranik2017domain,bynum2021decomposition}, which tightens bounds on variables by minimizing and maximizing variables subject to a relaxation of the problem. Rather than utilizing a linear outer approximation of a convex relaxation and solving LPs, tighter bounds could be achieved by solving the OBBT problems over a quadratic outer approximation of the convex relaxation. 
The broadly applicable principle of replacing linear outer approximations with quadratic outer approximations in DGO algorithms motivates this work to provide a methodology for the construction of tight quadratic outer approximators for twice-differentiable convex functions. We emphasize that our contributions are not restricted to any particular algorithm or technique, but apply generally to all contexts where outer approximation is relevant for rigorous optimality guarantees.

Quadratic outer approximators (a.k.a.``quadratic underestimators''; terms used interchangeably in this work) have been proposed in the literature and utilized in solving optimization problems. \citet{buchheim2013quadratic} solve general convex integer programs by using a point-wise maximum quadratic surrogate for nonlinear functions. In their distinct approach the authors construct quadratic outer approximators for a specific class of functions where the Hessian, $Q$, of a second-order Taylor series approximation is selected to satisfy $\nabla^2 f(x) \succcurlyeq Q$ for all $x \in \mathbb{R}^n$. While identifying a $Q$ that satisfies this condition a priori is not in general applicable to nonlinear functions, the underestimator possesses the property that the same Hessian, $Q$, can be reused for any point of construction in the domain without modification; once identified, construction is cheap and tighter than linear, but not as tight as possible. The authors acknowledge the difficulty of finding a suitable $Q$ for general nonlinear functions and suggest determining $Q$ for additional specific classes of functions as future work~\citep{buchheim2013quadratic}.

\citet{su2018improved} construct quadratic outer approximators for use in outer approximation algorithms for convex MINLP problems, where the cuts added to the master problem are convex quadratic instead of linear. Rather than selecting the same Hessian for the second-order term in a Taylor series approximation, the authors propose modifying the Hessian uniquely for each point of construction, thereby achieving tighter quadratic outer approximators. Their modified Taylor series approximation at a point of construction $x_0$ takes the form $f(x_0) + \nabla f(x_0)(x-x_0) + \frac{1}{2}(x-x_0)^\top \alpha\nabla^2 f(x_0)(x-x_0)$, where the scalar $\alpha$ modifies the second-order term to achieve underestimation over a pre-specified box domain $\mathcal{B}$.
Acknowledging the difficulty of solving the non-convex optimization problem (\ref{eq:suetalopt}) to determine the tightest possible value for $\alpha$, the authors observe that, in the case when the gradient of the objective function in (\ref{eq:suetalopt}) is coordinate-wise monotonic (i.e., each component of the gradient is monotonic over the feasible space), then the tightest possible choice for $\alpha$ (and the solution to (\ref{eq:suetalopt})) lies at a vertex of the box domain. However, while providing a procedure for identifying the tightest $\alpha$ for functions that satisfy this coordinate-wise monotonicity property,~\citet{su2018improved} leave open the problem of determining the tightest $\alpha$ for general nonlinear convex functions. 
\begin{equation}
\label{eq:suetalopt}
\alpha = \min_{x \in \mathcal{B} \setminus \{x_0\}} \frac{f(x) - f(x_0) - \nabla f(x_0)(x-x_0)}{\frac{1}{2}(x-x_0)^\top \nabla^2 f(x_0)(x-x_0)}
\end{equation}

\citet{olama2023distributed} consider an outer approximation algorithm that employs a quadratic underestimator of the form $f(x_0) + \nabla f(x_0)(x-x_0) + \frac{m}{2}(x-x_0)^\top (x-x_0)$, where $f$ is strongly convex given a suitable choice of $m$.
Closely resembling the approach proposed in~\citet{buchheim2013quadratic}, this methodology provides a quadratic underestimator that applies to any point of construction for a function, but also possesses the drawback that computing $m$ is only straightforward for a restricted class of functions (e.g., convex quadratic functions in problems of learning and control). Integrated into the outer approximation algorithm, the authors claim that an advantage of their quadratic outer approximator over the one proposed by~\citet{su2018improved} is that, by utilizing the strong convexity parameter, the Hessian of the quadratic constraint included in the outer approximation master problem is a diagonal matrix, and thereby less dense and more efficient to solve by available solvers~\citep{olama2023distributed}. Finally, the authors assert that integrating the quadratic outer approximation of the objective function of the learning and control type of problems studied (e.g., classification, optimal and predictive control, regression) computes tighter lower bounds for the outer approximation algorithm and facilitates faster convergence.

Recent work by~\citet{streeter2022automatically} develops the theory to construct polynomial under- and over-estimators of functions using so-called Taylor polynomial enclosures. Additionally, the authors present a detailed software implementation for their numerical procedure, named ``AutoBound'', which can construct quadratic under- and over-estimators for general non-convex functions, but provides no guarantee on the convexity of the resulting estimators. Furthermore, ~\citet{streeter2023sharp} show that the analytic expressions, rather than numerical procedures, can generate the \textit{tightest} estimators for univariate functions composed of functions with a monotonic second derivative (in the case of using quadratics as estimators). Considered in the context of the quadratic underestimator proposed in~\citet{su2018improved}, for univariate functions composed of functions with a monotonic second derivative, these analytic expressions identify  \textit{a priori} which vertex is the solution of~(\ref{eq:suetalopt}) and determine the tightest scaling parameter $\alpha$.

Despite all this recent research activity, we highlight that, to the best of our knowledge, the literature lacks a procedure to construct tight \textit{convex} quadratic underestimators of twice-differentiable convex functions in general. To this end, we provide the following distinct contributions:
\begin{itemize}
	\item We set forth a cutting-plane algorithm that numerically determines the tightest value of $\alpha$ in the underestimator proposed by~\citet{su2018improved} for general twice-differentiable convex functions of low dimensions.
	\item We provide qualitative and quantitative evidence of the quality of the quadratic underestimators and demonstrate the efficiency of the cutting-plane algorithm on a set of representative functions extracted from optimization benchmark libraries.
	\item We extend the cutting-plane algorithm to generate even tighter quadratic underestimators of functions present in optimization problems by allowing overestimation in infeasible regions, as informed by linear constraints in those problems.
\end{itemize}

The rest of the paper is organized as follows. In Section~\ref{sec:Methodology}, we present our approach for constructing tight quadratic underestimators of general twice-differentiable convex functions, while in Section~\ref{sec:computational_results} we present results on a library of convex functions. In Section~\ref{sec:externalLinearConstraints}, we show how to construct even tighter underestimators in the context of optimization problems with linear constraints, illustrating this extension with a representative example. Finally, in Section~\ref{sec:conclusions}, we present some conclusions from our work.

\section{Methodology} \label{sec:Methodology}
For underestimating a twice-differentiable convex function, $f : \mathbb{R}^n \to \mathbb{R}$, let $\mathcal{B} := \{\x \in \mathbb{R}^n : x^L_i \le x_i \le x^U_i \ \forall i = 1,...,n\}$ denote the box defined by given bounds $x^L_i$ and $x^U_i$ on the components of $\x$, and let $\x_0 \in \mathcal{B}$ be a point at which to construct a quadratic underestimator, 
\begin{equation}
\label{eq:quadform}
q(\x; \alpha, \x_0) := f(\x_0) + \nabla f(\x_0)(\x-\x_0) + \frac{1}{2}(\x-\x_0)^\top \alpha\nabla^2 f(\x_0)(\x-\x_0),
\end{equation}
where $\alpha$ is a scaling parameter selected to satisfy $q(\x; \alpha, \x_0) \le f(\x)$ for all $\x \in \mathcal{B}$. We note that herein and in the remainder of the paper, we denote vectors in bold.

Adopting the second-order Taylor series approximation at $\x_0$ for the quadratic underestimator~(\ref{eq:quadform}), as presented in~\citet{su2018improved}, we seek to solve the optimization problem~(\ref{eq:quadopt1}). Here, we note that $\alpha$ is necessarily bounded from above by~$1$, or else the Taylor series approximation would overestimate in the local vicinity of the point of construction.
\begin{equation}
\label{eq:quadopt1}
\begin{array}{cl}
 \max\limits_{\alpha \in [0,1]} & \alpha \\ 
 \text{s.t.}             & \min\limits_{\x \in B} \left\{f(\x) - q(\x, \alpha; \x_0)\right\} \ge 0
\end{array}
\end{equation} 

For a fixed value of $\alpha$, the minimization problem defining the feasible set of (\ref{eq:quadopt1}) is a difference of convex (d.c.) function minimization problem that, by utilizing a partial epigraph reformulation, can be cast as a concave minimization problem over a convex feasible set, as displayed in~(\ref{eq:quadoptform1}). 
\begin{equation}
\label{eq:quadoptform1}
\begin{array}{cl}
 \min\limits_{\x \in \mathcal{B}, t \in \mathbb{R}} & t - q(\x; \alpha, \x_0) \\ 
 \text{s.t.}             & f(\x) - t \le 0
\end{array}
\end{equation}

Various approaches exist in the literature for solving~(\ref{eq:quadoptform1}), including outer approximation, polyhedral annexation, successive underestimation, and branch-and-bound~\citep{horst2013handbook}. We implement an outer approximation scheme utilizing a cutting-plane algorithm to solve~(\ref{eq:quadoptform1}). Exploiting the monotonicity of (\ref{eq:quadform}) as a function of $\alpha$, as detailed in Observation~\ref{obs:mono}, we maintain a set of vertices corresponding to an outer approximation of the convex feasible region of~(\ref{eq:quadoptform1}), which we evaluate against the objective function of~(\ref{eq:quadoptform1}) to compute a bound on the overestimation of the function by the quadratic for a particular value of $\alpha$.
We then iteratively refine our outer approximation vertex set using a cutting-plane algorithm to separate vertices until the bound computed by (\ref{eq:quadoptform1}) is within some user-specified tolerance, $\varepsilon$. To guarantee convergence of the algorithm and simultaneously determine the tightest parameter $\alpha$, i.e., the solution to~(\ref{eq:quadopt1}), we initialize the cutting-plane algorithm with $\alpha = 1$ and decrement the value of $\alpha$ at each enumerated vertex of the cutting-plane algorithm if (\ref{eq:quadform}) overestimates the function. The cutting-plane algorithm is presented as a diagram in Figure~\ref{fig:cutting_plane_algo}, while detailed explanations of each step are provided in Section~\ref{sec:solution_approach_algorithm}. 

\begin{obs}
	\label{obs:mono} 
	The quadratic underestimator in (\ref{eq:quadform}) monotonically decreases as $\alpha$ decreases; that is, $q(\x; \alpha_1, \x_0) \le q(\x; \alpha_2, \x_0)$ for all $\alpha_1 < \alpha_2$.
\end{obs}

\begin{proof}
By construction, $\nabla^2 f(\x_0) \succcurlyeq 0$, which implies that $\frac{1}{2}(\x-\x_0)^\top \nabla^2 f(\x_0) (\x-\x_0) \ge 0$ for all $\x \in \mathbb{R}^n$.
Let two values for parameter $\alpha$, namely $\alpha_1 \in [0,1)$ and $\alpha_2 \in (0,1]$ such that $\alpha_1 < \alpha_2$. We thus have:
\[\begin{array}{rl}
\alpha_1 < \alpha_2 & \Longrightarrow \frac{\alpha_1}{2}(\x-\x_0)^\top \nabla^2 f(\x_0) (\x-\x_0) \le \frac{\alpha_2}{2}(\x-\x_0)^\top \nabla^2 f(\x_0) (\x-\x_0) \\ 
 & \Longleftrightarrow q(\x; \alpha_1, \x_0) \le q(\x; \alpha_2, \x_0),
\end{array}\]
noting that equality holds only at points $\x^* \in \mathbb{R}^n$ for which $\frac{1}{2}(\x^*-\x_0)^\top \nabla^2 f(\x_0)(\x^*-\x_0) = 0$.
\end{proof}

\subsection{The Cutting-Plane Algorithm}\label{sec:solution_approach_algorithm}
\label{par:algo}
 \begin{figure}[h]
    \centering
    \resizebox{0.6\textwidth}{!}{
     \begin{tikzpicture}[node distance=4cm, scale=0.5]
      \node (initialValues) [algorithmStep] {Initialize $\alpha = 1$};
      \node (toleranceCheck) [decision, below of=initialValues, yshift=1.5cm] { LB $> -\varepsilon$?};
      \node (STOP) [algorithmStep2, right of=toleranceCheck, xshift=1.5cm] {STOP, Return $\alpha$};
      \node (separatingCut) [algorithmStep, below of=toleranceCheck, yshift=1.25cm] {Separate Vertex};
      \node (enumerateVertices) [algorithmStep, below of=separatingCut, yshift=2.0cm] {Enumerate Vertices \citep{chen1991line}};
      \node (checkVertices) [decision, below of=enumerateVertices, yshift=0.5cm] {Overestimation at Vertices?};
      \node (evaluateVertices) [algorithmStep, below of=checkVertices, yshift=1.0cm] {Evaluate Objective Function at Vertices};
      \node (decreaseAlpha) [algorithmStep2, right of=checkVertices, xshift=1.5cm] {Decrease $\alpha$};
      \node (updateLB) [algorithmStep2, left of=separatingCut, xshift=-1.0cm, yshift=-1cm] {Update LB};
      \draw [arrow] (initialValues) -- (toleranceCheck);
      \draw [arrow] (toleranceCheck) -- node[anchor=south, xshift=0.5cm, yshift=-0.25cm] {N} (separatingCut);
      \draw [arrow] (separatingCut) -- (enumerateVertices);
      \draw [arrow] (enumerateVertices) -- (checkVertices);
   	  \draw [arrow] (checkVertices) -- node[anchor=south, xshift=0.5cm, yshift=-0.25cm] {N} (evaluateVertices);
      \draw [arrow] (toleranceCheck) -- node[anchor=east, yshift=0.25cm] {Y} (STOP);
      \draw [arrow] (evaluateVertices.south)  |- ($(evaluateVertices.south) + (0,-1.5cm)$) -| (updateLB.south);
      \draw [arrow] (updateLB.north) node[anchor=west] {} |- (toleranceCheck.west);
      \draw [arrow] (checkVertices) -- node[anchor=east, yshift=0.25cm] {Y} (decreaseAlpha);
      \draw [arrow] (decreaseAlpha.south) |- ($(evaluateVertices.east)$);
     \end{tikzpicture}
    } 
    \caption{Diagram illustrating the cutting-plane algorithm.}
    \label{fig:cutting_plane_algo}
 \end{figure}
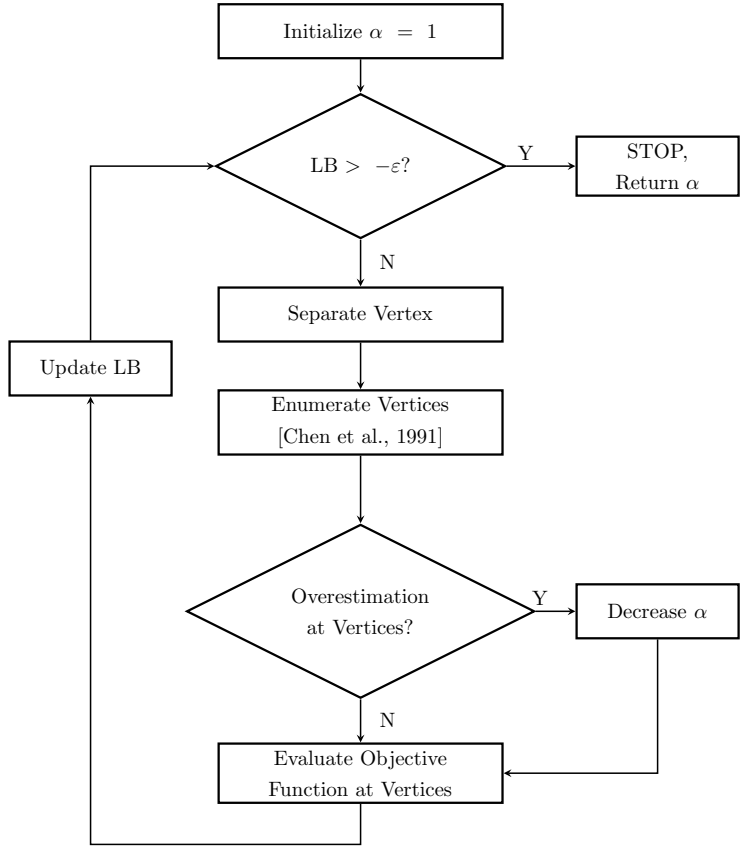

Let $\mathcal{U}^{(k)}$ be the set of vertices at iteration $k$ that outer approximates the epigraph of $f$, and let $\mathcal{A}^{(k)} \subseteq \mathcal{U}^{(k)}$ be the set of vertices to be evaluated against the objective in~(\ref{eq:quadoptform1}). Let $\mathcal{B}$ be the set of points in the continuous box domain defined by the bounds on variables~$\x$. Let $\mathcal{V}$ be the $2^n$ vertices of $\mathcal{B}$, where $n$ is the dimension of vector $\x$. 

\begin{enumerate}
\item\label{main_algo:init}
\textit{Initialize:} Set $k \gets 0$, $\alpha^{(0)} \gets 1$, and choose tolerance $\varepsilon$.\\
$\mathcal{A}^{(0)}, \mathcal{U}^{(0)} \gets \left\{ (\x,t): \x \in \mathcal{V}, t \in \{t^L, t^U\} \right\}$, where $t^L := \min_{\x \in \mathcal{B}}f(\x)$ and $t^U := \max_{\boldsymbol{v} \in \mathcal{V}}f(\boldsymbol{v})$.\\
Select a point $(\x_p, t_p) \in \mathcal{B} \times [t^L,t^U]$ such that $f(\x_p) < t_p$. \\
Initialize the vertex enumeration data structures as described by~\citet{chen1991line}.

\item\label{main_algo:convergence}
\textit{Check Convergence:} $LB = \min\limits_{(\x, t) \in \mathcal{A}^{(k)}} t - q(\x; \alpha^{(k)}, \x_0)$.\\
If $LB > -\varepsilon$, then STOP, since $f(\x) - q(\x; \alpha^{(k)}, \x_0) \ge -\varepsilon$ for all $\x \in \mathcal{B}$, and return $\alpha^{(k)}$.

\item\label{main_algo:separate_min_vertex}
\textit{Separate Vertex:} $\boldsymbol{u} \gets \argmin\limits_{(\x, t) \in \mathcal{A}^{(k)}} t - q(\x; \alpha^{(k)}, \x_0)$. $\boldsymbol{w} \gets \lambda \boldsymbol{u}$, where $\lambda$ solves $f(\lambda \x_u + (1-\lambda) \x_p) - (\lambda t_u + (1-\lambda) t_p) = 0$. Compute separating hyperplane, $\ell(\x, t) := f(\boldsymbol{w}) + \nabla f(\boldsymbol{w})((\x,t) - \boldsymbol{w}) \le 0$.

\item\label{main_algo:enumerate_vertices}
\textit{Enumerate Vertices:} $\mathcal{U}^{(k+1)} \gets (\mathcal{U}^{(k)}\cup \mathcal{H}^+) \backslash \mathcal{H}^- $, where $\mathcal{H}^- := \{(\x,t) \in \mathcal{U}^{(k)} : \ell(\x,t) > 0\}$, and $\mathcal{H}^+$ := $\text{extr}\left(\{(\x,t) \in \text{conv}(\mathcal{U}^{(k)}) : \ell(\x ,t) \le 0\}\right)$, where extr($\mathcal{S}$) is the set of extreme points of a set $\mathcal{S}$.

\item\label{main_algo:evaluate_underestimation}
\textit{Evaluate Underestimation:} \\
$\alpha^{(k)} \gets \min\limits_{(\x_v, t_v) \in \mathcal{H}^+} \left\{\frac{2\left(f(\x_v) - (f(\x_0) + \nabla f(\x_0)(\x_v-\x_0)\right)}{(\x_v-\x_0)^\top \nabla^2f(\x_0)(\x_v-\x_0)} \ : \ f(\x_v) - q(\x_v; \alpha^{(k)}, x_0) < -\varepsilon \right\}$.

\item\label{main_algo:loop}
\textit{Update $\mathcal{A}:$} $\mathcal{A}^{(k+1)} \gets \{(\x, t) \in (\mathcal{A}^{(k)} \cup \mathcal{H}^+)\backslash \mathcal{H}^- : t - q(\x; \alpha^{(k)}, \x_0) < -\varepsilon\}$. 

\item\label{main_algo:continue}
\textit{Iterate:} $k \gets k +1$. Go to Step~\ref{main_algo:convergence}.
\end{enumerate}

We make the following remarks concerning the cutting-plane algorithm.
\begin{rem}
	\label{rem:cut0}
	For a valid underestimator of form~(\ref{eq:quadform}), the optimal value of~(\ref{eq:quadoptform1}) is exactly 0 and is attained at $\x = \x_0$. The main task of the cutting-plane algorithm is to decrease $\alpha$ and increase the lower bound until the quadratic is guaranteed to overestimate the function by at worst $\varepsilon$ over the domain.
\end{rem}

\begin{rem}
	\label{rem:cut1}
	In Step~\ref{main_algo:init}, $t_L = \min_{\x \in \mathcal{B}}f(\x)$ is a convex minimization problem, which is robustly and efficiently solvable for small $n$ using readily available convex optimization algorithms, while $t_U = \max_{\x \in \mathcal{B}}f(\x)$ maximizes a convex function over a convex feasible set, which attains its maximum at a vertex. Hence, $t_U = \max_{\boldsymbol{v} \in \mathcal{V}}f(\boldsymbol{v})$.
\end{rem}

\begin{rem}
	\label{rem:cut2}
	The cutting-plane algorithm requires a point from the interior of the epigraph set, $\{(\x_p, t_p) \in \mathcal{B} \times [t^L,t^U] : f(\x_p) < t_p\}$. We select the center point of this domain as our interior point. 
\end{rem}

\begin{rem}
	\label{rem:cut3}
	The tolerance $\varepsilon$ can be adjusted to reflect the acceptable limit of overestimation and introduces a tradeoff between algorithm efficiency and underestimation correctness. For any function, subtracting the final lower bound from the quadratic provides an underestimator that is guaranteed to underestimate the function over the full domain.
\end{rem}

\begin{rem}
	\label{rem:cut4}
	In Step~\ref{main_algo:convergence}, we only need to evaluate the objective of (\ref{eq:quadoptform1}) for each vertex in $\mathcal{A}^{(k)}$ (rather than $\mathcal{U}^{(k)}$) because Observation~\ref{obs:mono} implies that any vertex where $t - q(\x; \alpha^{(k)}, \x_0) >= -\varepsilon$, i.e., the underestimating vertices in $\mathcal{U}^{(k)}$, will never satisfy $t - q(\x; \alpha^{(k)}, \x_0) < -\varepsilon$ at any future iteration of the algorithm. Hence, using $\mathcal{A}^{(k)}$ reduces the number of computations required for each iteration and thereby expedites the algorithm.
\end{rem}

\begin{rem}
	\label{rem:cut5}
	Step~\ref{main_algo:separate_min_vertex} requires solving the equation $f(\lambda \x_u + (1-\lambda) \x_p) - (\lambda t_u + (1-\lambda) t_p) = 0$ for computing scalar $\lambda \in (0,1)$. This can be achieved efficiently using the bisection method.

\end{rem}

\begin{rem}
	\label{rem:cut6}
	Step~\ref{main_algo:enumerate_vertices} enumerates vertices using the procedure outlined in~\citet{chen1991line}. We note that this procedure must maintain a set of vertices that are non-degenerate. For a cutting plane generated by the algorithm, defined as $A\boldsymbol{v} + b \le 0$, an existing vertex $\boldsymbol{v}^*$ becomes degenerate if $A\boldsymbol{v}^* + b = 0$. Whenever such degeneracy occurs, we eliminate it by applying a small perturbation to $b$, similar to the approach described in~\citet{chen1991line}.
\end{rem}

\begin{rem}
	\label{rem:cut7}
	In Step~\ref{main_algo:evaluate_underestimation}, we note the computation $\alpha^{(k)} = \frac{2\left(f(\x_v) - (f(\x_0) + \nabla f(\x_0)(\x_v-\x_0)\right)}{(\x_v-\x_0)^\top \nabla^2f(\x_0)(\x_v-\x_0)}$ is not defined where $(\x_v-\x_0)^\top \nabla^2f(\x_0)(\x_v-\x_0) = 0$ and highlight that Step~\ref{main_algo:evaluate_underestimation} only employs $\alpha^{(k)} = \frac{2\left(f(\x_v) - (f(\x_0) + \nabla f(\x_0)(\x_v-\x_0)\right)}{(\x_v-\x_0)^\top \nabla^2f(\x_0)(\x_v-\x_0)}$ when $f(\x_v) - q(\x_v; \alpha^{(k)}, x_0) < -\varepsilon$. In the case that $(\x_v-\x_0)^\top \nabla^2f(\x_0)(\x_v-\x_0) = 0$, the quadratic underestimator at $\x_v$ is contained in the first-order Taylor series approximation at $\x_0$, which is guaranteed to underestimate $f$ at $\x_v$ due to the convexity of $f$. Hence, at any point $\x_v$ where $(\x_v-\x_0)^\top \nabla^2f(\x_0)(\x_v-\x_0) = 0$, the quadratic underestimates $f$ and the update to the parameter, i.e., $\alpha^{(k)} = \frac{2\left(f(\x_v) - (f(\x_0) + \nabla f(\x_0)(\x_v-\x_0)\right)}{(\x_v-\x_0)^\top \nabla^2f(\x_0)(\x_v-\x_0)}$, will never be invoked.
\end{rem}

\begin{rem}
	\label{rem:cut8}
	The cutting-plane algorithm uses a tolerance on the lower bound of the objective of~(\ref{eq:quadoptform1}) as its convergence criterion. Alternatively, a time limit could be given to the algorithm as a termination criterion, where the lower bound upon termination could be subtracted from the final quadratic expression.
\end{rem}

\subsubsection{Algorithm Convergence}
For a proof on convergence, we refer the reader to~\citet{hoffman1981method}, where the convergence of a cutting-plane algorithm for solving a concave minimization problem over a convex set (e.g.,~(\ref{eq:quadoptform1})) is proved. In particular, the authors of that work prove three lemmas as a foundation for their convergence proof: Lemma~1 in \citet{hoffman1981method} states that the sequence of lower bounds of the objective function $\{\phi(x^k)\}$ is monotonically (but not necessarily strictly) increasing, while Lemmas~2 and~3 in \citet{hoffman1981method} deal with the procedure for constructing separating hyperplanes (Step~\ref{main_algo:separate_min_vertex}).
As we use an identical procedure for separating hyperplanes, the proofs for Lemmas~2 and~3 in~\citet{hoffman1981method} apply to our algorithm, with the inconsequential difference that the concave minimization problem~(\ref{eq:quadoptform1}) has a single nonlinear, convex constraint defining the feasible region (rather than multiple, as in the more general case proved in~\citet{hoffman1981method}). We adapt and prove Lemma~1 from~\citet{hoffman1981method} for our specific case, where in Algorithm~\ref{par:algo} we dynamically change the concave objective function by updating the value of $\alpha^{(k)}$ at Step~\ref{main_algo:evaluate_underestimation}.

\begin{lem}{1}[adapted from~\citet{hoffman1981method}] \label{lem:hoffman1}
	The sequence of lower bounds $\{\gamma(x^k, \alpha^{(k)}, \x_0) := \min\limits_{x^k \in \mathcal{A}^{(k)}} t - q(\x^k; \alpha^{(k)}, \x_0)\}$ is monotonically increasing.
\end{lem}

\begin{proof}
At each consecutive iteration of Algorithm~\ref{par:algo}, the feasible set over which $\gamma(x^k, \alpha^{(k)}, \x_0)$ is minimized contracts, i.e., if $\mathcal{S}^{(k)} := \text{conv}(\mathcal{U}^{(k)})$, where $\mathcal{U}^{(k)}$ is the set of extreme points comprising the outer approximation of the convex feasible set $\{(\x, t) : f(\x) - t \le 0\}$ at iteration $k$, then $\mathcal{S}^1 \supseteq \mathcal{S}^2 \supseteq \ldots \supseteq \mathcal{S}^{k}$, thereby ensuring that $\gamma(\x^k, \alpha^{(k)}, \x_0)$ monotonically increases in the iterations where $\alpha^{(k)}$ is not updated due to the lack of eligible vertices in Step~\ref{main_algo:evaluate_underestimation}.
Furthermore, at any iteration where Step~\ref{main_algo:evaluate_underestimation} proceeds and $\alpha^{(k)}$ is updated, Observation~\ref{obs:mono} implies that $q(\x; \alpha^{(k)}, \x_0)$ will monotonically decrease, and thereby monotonically increase $t - q(\x; \alpha^{(k)}, \x_0)$. Hence, at all iterations of Algorithm~\ref{par:algo}, $\gamma(\x^k, \alpha^{(k)}, \x_0)$ monotonically increases.
\end{proof}

Considering Lemma~\ref{lem:hoffman1} together with Lemmas~2 and~3 from~\citet{hoffman1981method}, which hold for Algorithm~\ref{par:algo} without further qualification, the convergence of Algorithm~\ref{par:algo} now follows directly from Theorem~2 in~\citet{hoffman1981method}.
Highlighting a distinct property of our algorithm from the one in~\citet{hoffman1981method}, we reiterate the fact that, for a properly constructed underestimator $q(\x; \alpha, \x_0)$, the global minimum value of~(\ref{eq:quadoptform1}) is 0 (attained at $\x=\x_0$), and so the algorithm will terminate only when all of the vertices comprising the outer approximation of the convex feasible set evaluate the objective of~(\ref{eq:quadoptform1}) at a value that is larger than the user specified tolerance, $-\varepsilon$, which equivalently guarantees the quadratic overestimates the function by at most $\varepsilon$.

\subsection{Properties of Generated Underestimators}
After setting forth a preliminary assumption, we establish two important properties of the quadratic underestimators generated by Algorithm~\ref{par:algo}.

\begin{assump}
	\label{assump:decrementAlpha}
	 Given a function $f(\x) : \mathbb{R}^n \to \mathbb{R}$ with known bounds on $\x$, we assume that, during the execution of Algorithm~\ref{par:algo}, $\alpha$ is decremented (Step~\ref{main_algo:evaluate_underestimation}) \textit{at least} once.
\end{assump}

We note that Assumption~\ref{assump:decrementAlpha} generally holds in practice, but we identify two cases where $\alpha$ may not be decremented even once.
The first is the case where the tolerance for the algorithm is selected arbitrarily large so that the condition for decrementing $\alpha$ in Step~\ref{main_algo:evaluate_underestimation} never applies. In this situation, one is advised to select a smaller tolerance and rerun the algorithm.
The other case is when $\x_0$ is at a point of $f$ where $\nabla^2 f(\x) \succcurlyeq \nabla^2 f(\x_0)$ for all $\x \in $dom$(f)$, which implies that $q(\x; \alpha, x_0)$ with $\alpha = 1$ underestimates $f(\x)$ over all points in $\mathcal{B}$.

\begin{property}
	\label{prop:coincide}
	Under Assumption~\ref{assump:decrementAlpha}, the quadratic underestimator defined by $\alpha$ will have \textit{at least} two points that coincide with $f$.
\end{property} 

\begin{proof} 
The proof is straightforward. The first point of coincidence is $\x = \x_0$, by construction. The second point of coincidence is the last point where $\alpha$ was decremented, which is guaranteed to exist by Assumption~\ref{assump:decrementAlpha}, so that $\exists \x^{\prime} \ne \x_0 \in \mathcal{B}$ s.t. $q(\x^{\prime}; \alpha^{(k)}, \x_0) = f(\x^{\prime})$ by Step~\ref{main_algo:evaluate_underestimation} of Algorithm~\ref{par:algo}. Hence, at least two points of coincidence exist.
\end{proof}

\begin{property}
	\label{prop:tightestAlpha}
	The $\alpha$ returned by Algorithm~\ref{par:algo} is the largest value of $\alpha$ that (i) overestimates $f(\x)$ at all enumerated vertices by at most $\varepsilon$ and (ii) underestimates $f(\x)$ at all points where $\alpha$ was decremented (i.e., that satisfied the condition of Step~\ref{main_algo:evaluate_underestimation} during the execution of the algorithm).
\end{property}  

\begin{proof} 
We first state that, because LB $> -\varepsilon$ at the termination of Algorithm~\ref{par:algo} and by Step~\ref{main_algo:evaluate_underestimation} in Algorithm~\ref{par:algo}, the value of $\alpha$ returned satisfies the conditions in Property~\ref{prop:tightestAlpha}. We prove by contradiction that there is no larger $\alpha$. For any $\alpha^{\prime} > \alpha$, consider the last point at which $\alpha$ was decremented, $\x^{\prime}$, which exists by Assumption~\ref{assump:decrementAlpha}. Then, because of Lemma~1 and Remark~\ref{rem:cut7}, we have that $q(\x^{\prime}; \alpha^{\prime}, \x_0) > q(\x^{\prime}; \alpha, \x_0)$. In Step~\ref{main_algo:evaluate_underestimation} at $\x^{\prime}$, we have $q(\x^{\prime}; \alpha, \x_0) = f(\x^{\prime})$, and therefore $q(\x^{\prime}; \alpha^{\prime}, \x_0) > f(\x^{\prime})$, which contradicts the second condition in Property~\ref{prop:tightestAlpha} that $q(\x^{\prime}; \alpha, \x_0)$ underestimates $f(\x)$ at all points where $\alpha$ was decremented during the execution of the algorithm.
\end{proof}

\begin{rem}
	\label{rem:prop1}
	In the case where $q(\x; 1, \x_0)$ is a valid underestimator of $f$ and Assumption~\ref{assump:decrementAlpha} does not hold, we note that, although Property~\ref{prop:coincide} may not necessarily apply, Property~\ref{prop:tightestAlpha} does because the quadratic underestimates the function at all points in the domain ($\nabla^2 f(\x) \succcurlyeq \nabla^2 f(\x_0) \ \forall \x \in \mathbb{R}^n$).
\end{rem}

\begin{rem}
	\label{rem:prop2}
	In Step~\ref{main_algo:evaluate_underestimation} of Algorithm~\ref{par:algo}, we could update $\alpha$ and allow the quadratic to overestimate the function by the specified tolerance, i.e., $\alpha^{(k)} \gets \frac{2\left(f(\x_v) + \varepsilon - (f(\x_0) + \nabla f(\x_0)(\x_v-\x_0)\right)}{(\x_v-\x_0)^\top \nabla^2f(\x_0)(\x_v-\x_0)}$, and eliminate the second condition in Property~\ref{prop:tightestAlpha} requiring underestimation at all points where $\alpha$ was decremented. We make the following two comments regarding this choice: (i) if we included the tolerance in the update, then the final returned $\alpha$ would certainly overestimate the last decremented point by $\varepsilon$, thereby invalidating Property~\ref{prop:coincide} and (ii) the ultimate objective of the algorithm is to construct underestimators, so by using our update to $\alpha$ in Step~\ref{main_algo:evaluate_underestimation}, we ensure that the quadratic underestimates at every point where the criterion of Step~\ref{main_algo:evaluate_underestimation} holds.
\end{rem}

\subsection{Effect of Functional Form on Convergence}\label{subsec:path_cases}
We qualitatively characterize and provide some intuition for pathological functions that might lead our cutting-plane algorithm to slow convergence. The termination of Algorithm~\ref{par:algo} occurs in Step~\ref{main_algo:convergence} and requires that all of the vertices of the outer approximation of $f$, that is $\{(\x, t) \in U^{(k)}\}$, evaluate $t - q(\x; \alpha^{(k)}, \x_0) > -\varepsilon$. Consequently, if utilizing the scaled curvature at $\x_0$ creates large deviations between $f$ and the underestimator at points in the domain away from $\x_0$, then the vertices comprising the outer approximation of $f$ need less refinement in the cutting plane algorithm for $t - q(\x; \alpha^{(k)}, \x_0) > -\varepsilon$ to hold, thereby requiring fewer iterations of the cutting plane algorithm and faster convergence.
Additionally, the number of cutting planes required to outer approximate $f$ to a certain degree of precision also contributes to the rate at which the termination criterion is satisfied, as more cutting planes may need to be generated to reduce the difference between the outer approximation and the quadratic to within the required tolerance to satisfactorily prove underestimation. These two considerations, concisely reiterated as (i) the degree to which the function deviates from quadratic and (ii) how efficiently the function can be approximated polyhedrally, determine the number of cutting plane iterations required for convergence. We also note that, as the dimensionality of the domain of the function increases, nonlinear functions require many more cutting planes for precise polyhedral outer approximations. With this qualitative description, we provide illustrative examples in Figures~\ref{fig:pathological_i} and~\ref{fig:pathological_ii} that highlight the impact of~(i) and~(ii) on Algorithm~\ref{par:algo}.  

More specifically, Figure~\ref{fig:pathological_i} plots the 2D functions $f(\x) = x_1^2 + x_2^2$, $f(\x) = x_1^4 + x_2^4$ and $f(\x) = x_1^6 + x_2^6$ over the domain $\x \in [-1,1]^2$, for which using the point of construction $\x_0 = (0.5, 0.5)$ leads Algorithm~\ref{par:algo} to perform 1387, 45, and~36 iterations, respectively. Therefore, it can be implied that, the greater the dissimilarity from the underestimator form (i.e., quadratic) over the domain, the faster the algorithm's convergence.
Figure~\ref{fig:pathological_ii} displays three different \textit{quadratic} functions, each of which is its own convex quadratic underestimator. We highlight that the functions in~\ref{fig:path1}, \ref{fig:path2}, and~\ref{fig:path3} progressively become ``more linear'' in the sense that the minimum eigenvalue of the quadratic function in one principal direction approaches~$0$.
Applied on these functions with construction point $\x_0 = (0.5, 0.5)$, Algorithm~\ref{par:algo} requires 1387, 804, and~46 iterations, respectively, evidencing the importance of not only the deviation of the function from the final underestimator but also how precisely the function can be outer approximated using few cutting planes. In practice, quadratic functions can be identified early and omitted from the underestimation procedure. Nevertheless, these principles also apply generally to other functions that have topologies similar to quadratics and/or to functions that are challenging for hyperplanes to efficiently outer-approximate.  

 \begin{figure}[htb]
	\centering
	\includegraphics[width=0.6\textwidth]{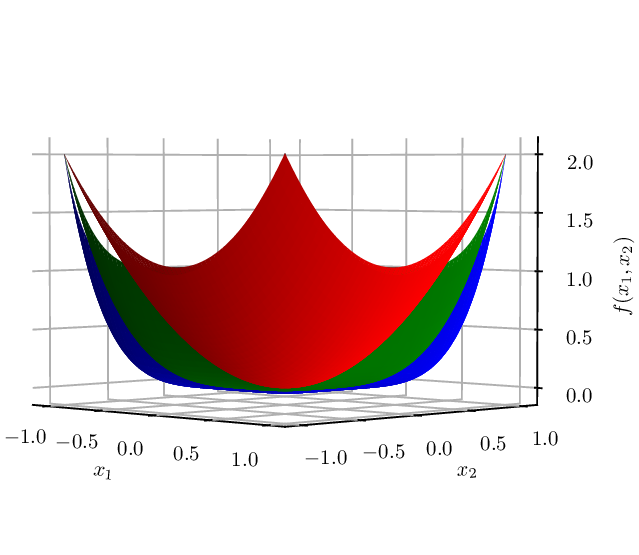}
	\caption{Cross-section of a hierarchy of convex functions that progressively deviate from a perfect quadratic: $f(\x) = x_1^2 + x_2^2$ (red), $f(\x) = x_1^4 + x_2^4$ (green), $f(\x) = x_1^6 + x_2^6$ (blue).}
	\label{fig:pathological_i}
\end{figure}

 \begin{figure}[htb]
	\centering
	\begin{subfigure}{0.32\textwidth}
	\includegraphics[width=\textwidth]{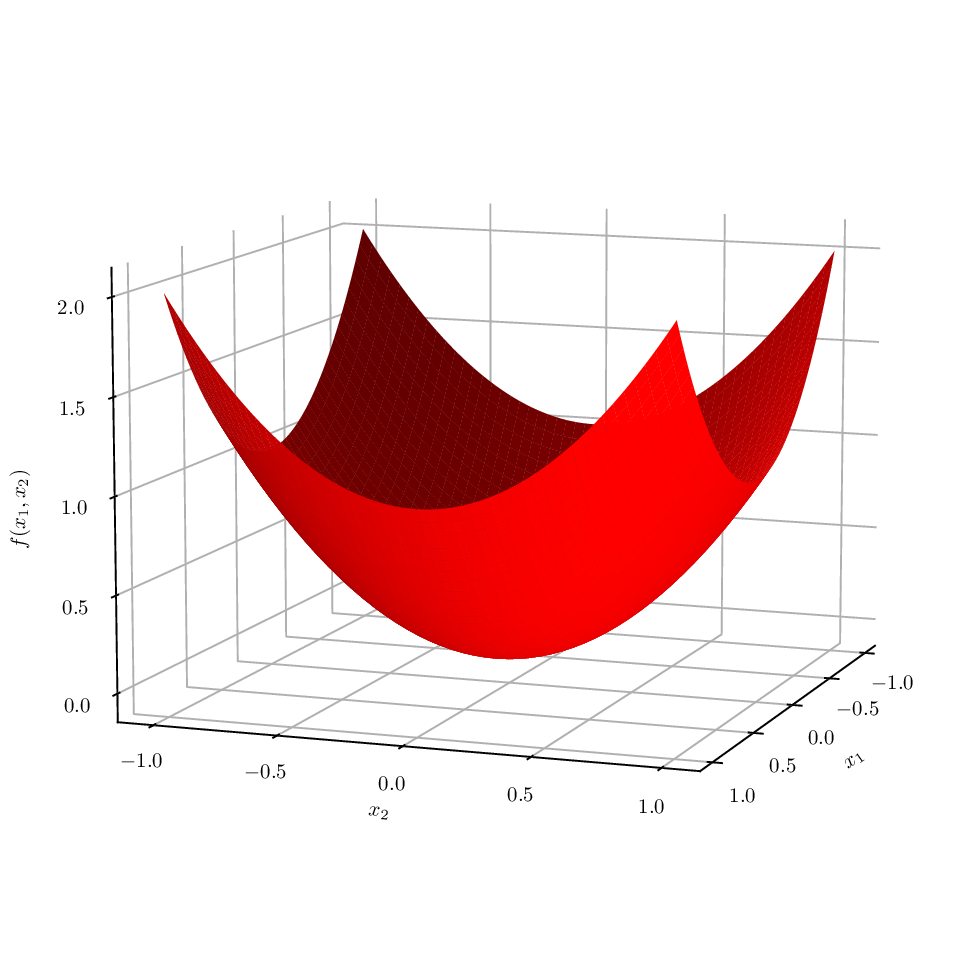}\hfill
	\caption{$f(\x) = x_1^2 + x_2^2$}
	\label{fig:path1}
	\end{subfigure}
	\hfill
	\begin{subfigure}{0.32\textwidth}
	\includegraphics[width=\textwidth]{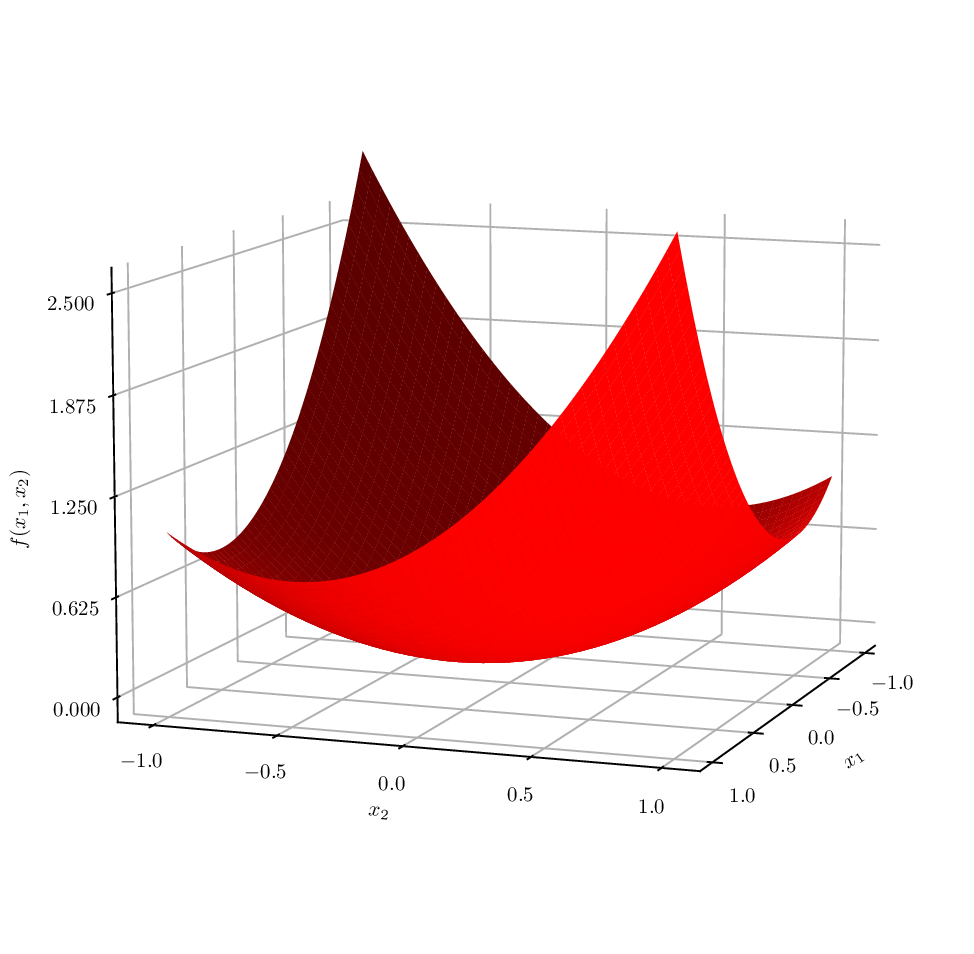}\hfill
	\caption{$f(\x) = x_1^2 + x_2^2 + x_1x_2$}
	\label{fig:path2}
	\end{subfigure}
	\hfill
	\begin{subfigure}{0.32\textwidth}
	\includegraphics[width=\textwidth]{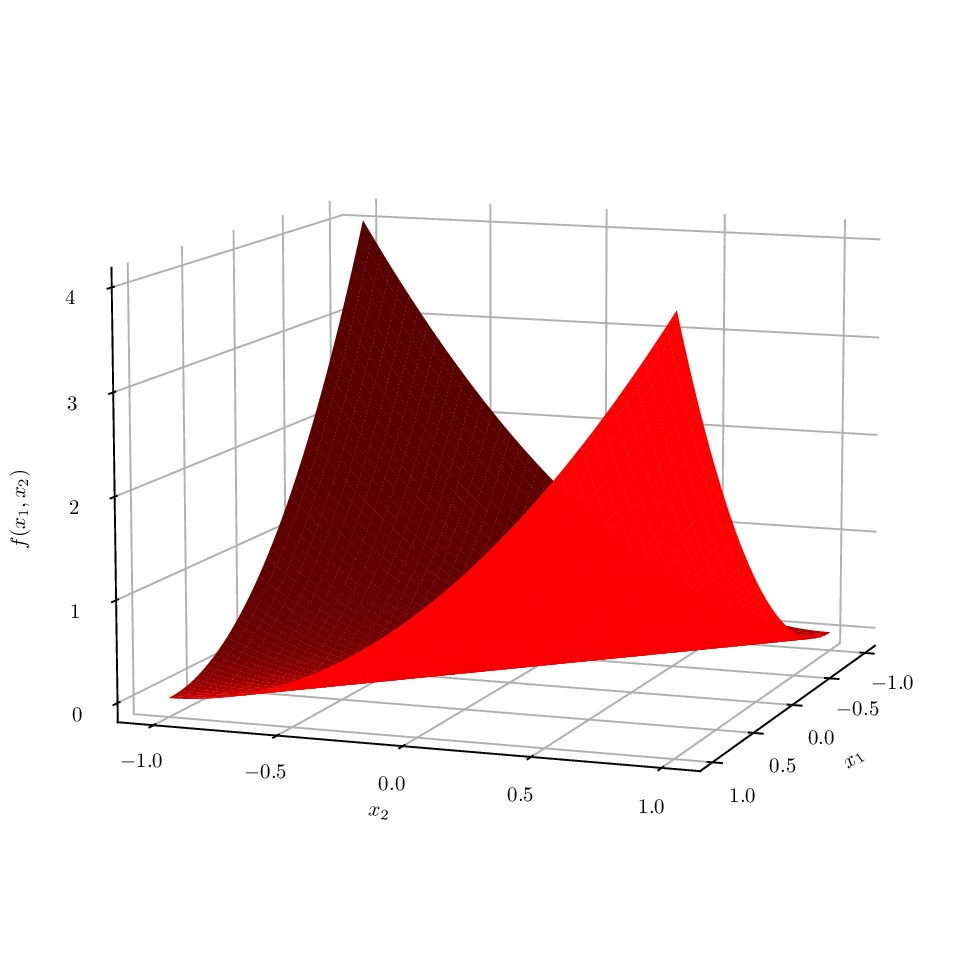}\hfill
	\caption{$f(\x) = x_1^2 + x_2^2 + 2x_1x_2$}
	\label{fig:path3}
	\end{subfigure}
	\caption{Hierarchy of convex quadratic functions that, from left to right, become progressively more linear in one principal direction.}
	\label{fig:pathological_ii}
\end{figure}

\section{Computational Results}\label{sec:computational_results}
\subsection{Computational Experiment Details}
For our computational study, we extract all the nonlinear functions from the 233 MINLPLib (\url{https://www.minlplib.org/}) problems using the filtering criteria \texttt{CONVEX==TRUE} and \texttt{ProblemType == \{MBNLP, MINLP\}} applied to the downloadable problem property data sheet. The filtering criteria ensure that, after relaxing the discrete variables, all nonlinear functions in the problems are convex. We instantiate a Pyomo~\citep{hart2011pyomo,hart2017pyomo} model of the problem instance and execute feasibility-based bounds tightening (fbbt)~\citep{belotti2010feasibility, puranik2017domain} using the fbbt implementation available in Pyomo with default parameter values to create realistic bounds for our quadratic underestimation procedure. We further analyze each separable nonlinear expression and decompose it into sub-expressions, each providing an additional nonlinear expression candidate for underestimation. In total, we identify over 9,000 candidate expressions defined in no larger than 4 dimensions; however, many problems have repeated expressions of similar functional forms, and we therefore reduce the set of expressions into a final set of 19 representative functional forms, defined in 1~to~4 dimensions. Tables~\ref{tab:1DconvexFunctionListMinlplibWithCOCONUT}--\ref{tab:4DconvexFunctionListMinlplibWithCOCONUT} document each expression selected from the MINLPLib, the problem instance it was extracted from, the constraint (or objective) as named in the downloadable \texttt{.py} instance file, the precise location of the expression in the constraint (or objective), and the bounds used for the quadratic underestimation procedure as informed by executing fbbt on the problem instance. We note that, while the functions listed in Tables~\ref{tab:1DconvexFunctionListMinlplibWithCOCONUT}--\ref{tab:4DconvexFunctionListMinlplibWithCOCONUT} are mathematically equivalent to the originals found in the MINLPLib problems, they are presented with some algebraic manipulations (e.g., constants and variables factored out of expressions). Here, we emphasize that the performance of our algorithm does not depend on the symbolic representation of the expression and, thus, these manipulations have no impact on our results.

While the functions resulting from the above process provide representative nonlinear functions for quadratic underestimation with different parameterizations that capture nonlinearity within functional forms, they do not include many known convex functional forms that arise in global optimization problems (e.g., sums of even powered polynomials). To that end, we augment our function library with functions extracted from the COCONUT library (\url{https://arnold-neumaier.at/glopt/coconut/Benchmark/Benchmark.html}), which is a compilation of optimization problems from three different benchmark libraries, namely GlobalLib, CUTE, and constraint satisfaction test problems (CSTP). From those, we specifically identify functional forms that do not appear in the convex MINLPLib test set, and include them in our study. Tables~\ref{tab:1DconvexFunctionListMinlplibWithCOCONUT}--\ref{tab:4DconvexFunctionListMinlplibWithCOCONUT} include these functions with the same descriptions as with the MINLPLib-originated ones, with the exception that the location of the expressions are in reference to the GAMS~\citep{bussieck2004general} file of the instance problem. We remark that many of the problems in the COCONUT library do not have bounds imposed on variables, and thus we assign bounds for variables ensuring that the nonlinearities are captured. All bounds assigned to variables are explicitly noted in Tables~\ref{tab:1DconvexFunctionListMinlplibWithCOCONUT}--\ref{tab:4DconvexFunctionListMinlplibWithCOCONUT}. From the COCONUT library, we extract an additional 12 functions, for a grand total of 31 functions for our quadratic underestimation computational study.

We scale the values of each function to lie within the interval $[-1,1]$ by using the scaling factor $1/\max\{|\min\limits_{\x \in \mathcal{B}}f(\x)|, |\max\limits_{\x \in \mathcal{B}} f(\x)|\}$, and set our convergence tolerance $\varepsilon = 1$e$-3$. Note that, since the tolerance is applied after such function scaling, it effectively acts as a relative tolerance for all functions.

For each function, we generate 5~points of construction using Latin hypercube sampling and construct quadratic underestimators at each point. We evaluate the quality of our underestimators using the tightness metric defined by Equation~(\ref{eq:metric}), which computes the fractional reduction of the hypervolume between a linear underestimator $\ell$ and the function $f$ that results as a consequence of using a quadratic (as opposed to the linear) underestimator $q$, constructed at the same point. In order to compute this metric, we discretize the space using 100$d$ Latin hypercube-sampled points, where $d$ is the dimensionality of the function, which we find to be sufficient for our study. 
\begin{equation}
	\label{eq:metric}
	M^{q(\x)}_{\ell(\x)} := \frac{\int_{\x \in \textrm{dom}(f)} \big(q(\x)-\ell(\x)\big)d\x}{\int_{\x \in \textrm{dom}(f)} \big(f(\x)-\ell(\x)\big)d\x}
\end{equation}

The computational study is executed on a Lenovo ThinkPad T490 laptop equipped with a 1.80GHz~Intel(R)~Core(TM)~i7-8565U~CPU on a Ubuntu~22.04 virtual machine with 8GB~RAM and 4 logical processors. We implement the algorithm in Python, acknowledging that efficiency gains could be achieved by utilizing a compiled language.  

\subsection{Results}
We report in Table~\ref{tab:convexMINLPLibWithCOCONUT} the average and range of the tightness metric, CPU time, and number of vertices enumerated for the nonlinear expressions, categorized by dimension. For 1D and 2D functions, we observe an average metric of 0.533 and 0.575, respectively, while for 3D and 4D functions we observe an average metric of 0.399 and 0.354, respectively. The diminishing quality of the underestimators in higher dimensions can be explained in part by the rigid modification of the second-order term, where each element of $\nabla^2 f(\x_0)$ is identically scaled, thus requiring the choice of $\alpha$ to be sufficiently low to underestimate in all principal directions, and thereby limited by the direction of least curvature. Notwithstanding, we observe on average 35\% tighter underestimators for 4-dimensional functions compared to linear underestimators. Table~\ref{tab:convexMINLPLibWithCOCONUT} also evidences the efficient construction of the quadratic underestimators. The majority of the underestimators are very efficient to construct, requiring less than 33~CPU~milliseconds~(ms) for functions in 3D and below. The longest time to construct an underestimator in the 4D case required 182~ms but, on average, functions in 4 dimensions required 51~ms for construction. The number of vertices generated by the algorithm shows an exponential increase in relationship to the dimension and manifests the main tractability limitation.
Finally, to visually appreciate the tightness of the produced underestimators, we provide some depictions of underestimators for 1D and 2D functions in Figures~\ref{fig:1DQuadUnderestimators} and~\ref{fig:2DQuadUnderestimators}, respectively.

 \begin{table}[htb]
\caption{Quadratic underestimator statistics, by function dimension.}
\label{tab:convexMINLPLibWithCOCONUT}
 \adjustbox{width=\linewidth}{
 \begin{tabular}{ c c c c c c c c c } 
  \toprule
  \multirow{2}{*}{Dimension} & \multirow{2}{*}{\# Functions} & \multirow{2}{*}{\# Underestimators}& \multirow{2}{*}{Avg. Metric} & \multirow{2}{*}{Range Metric}& \multirow{2}{*}{Avg. CPU (ms)} & \multirow{2}{*}{Range CPU (ms)} & \multirow{2}{*}{Avg. \# Vertices} & \multirow{2}{*}{Range \# Vertices}\\ [0.5ex]
  & & & & \\
  \midrule
  1 & 14 & 70 & 0.533 & 0.000--0.899 &  4 & 1--11  & 16.3 & 4--28  \\
  2 &  8 & 40 & 0.575 & 0.169--0.976 &  9 & 4--25  & 50.5 & 24--148 \\
  3 &  6 & 30 & 0.399 & 0.086--0.726 & 11 & 1--33  & 95.9 & 16--316 \\
  4 &  3 & 15 & 0.354 & 0.063--0.776 & 51 & 2--182 & 533.1 & 32--1951 \\ 
  \bottomrule
 \end{tabular}}
 \end{table}

 \begin{figure}[htb]
	\centering
	\includegraphics[width=.32\textwidth]{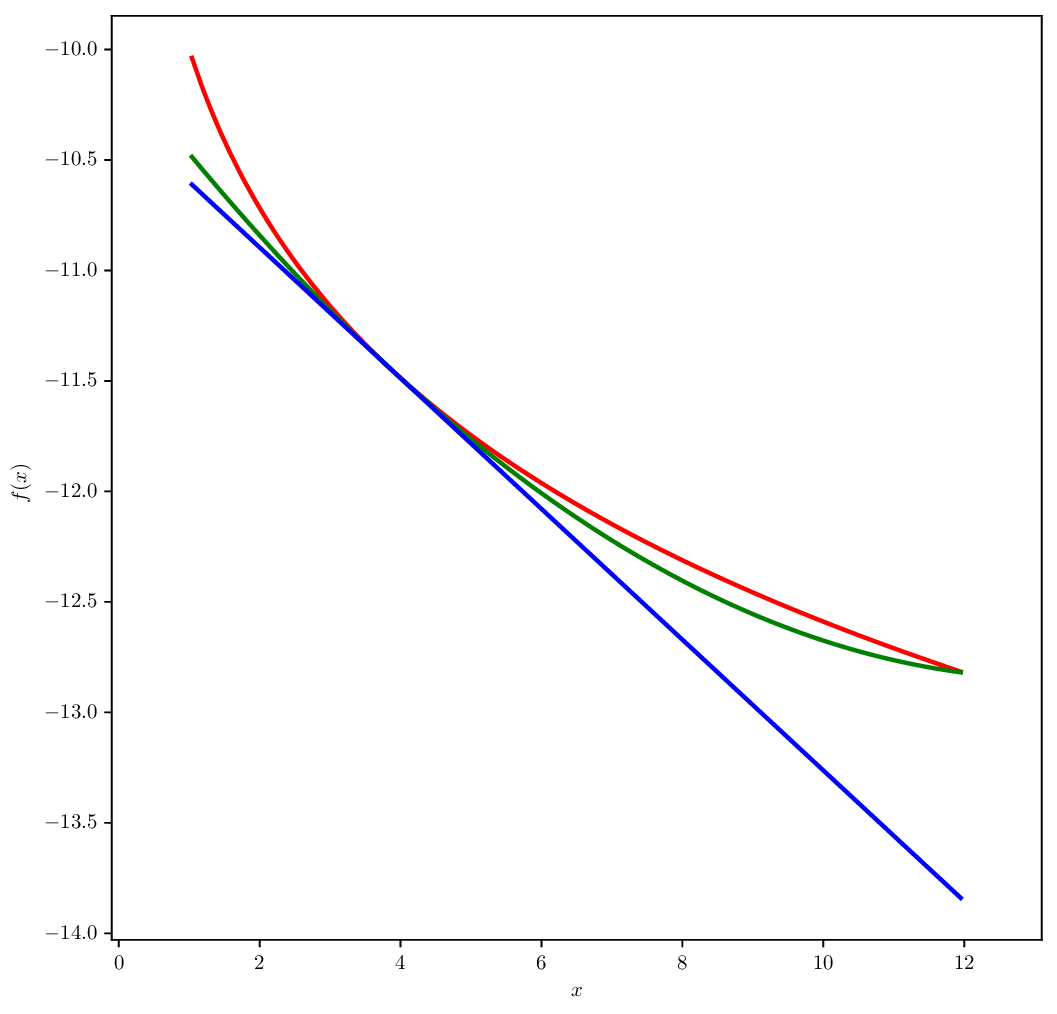}\hfill
	\includegraphics[width=.32\textwidth]{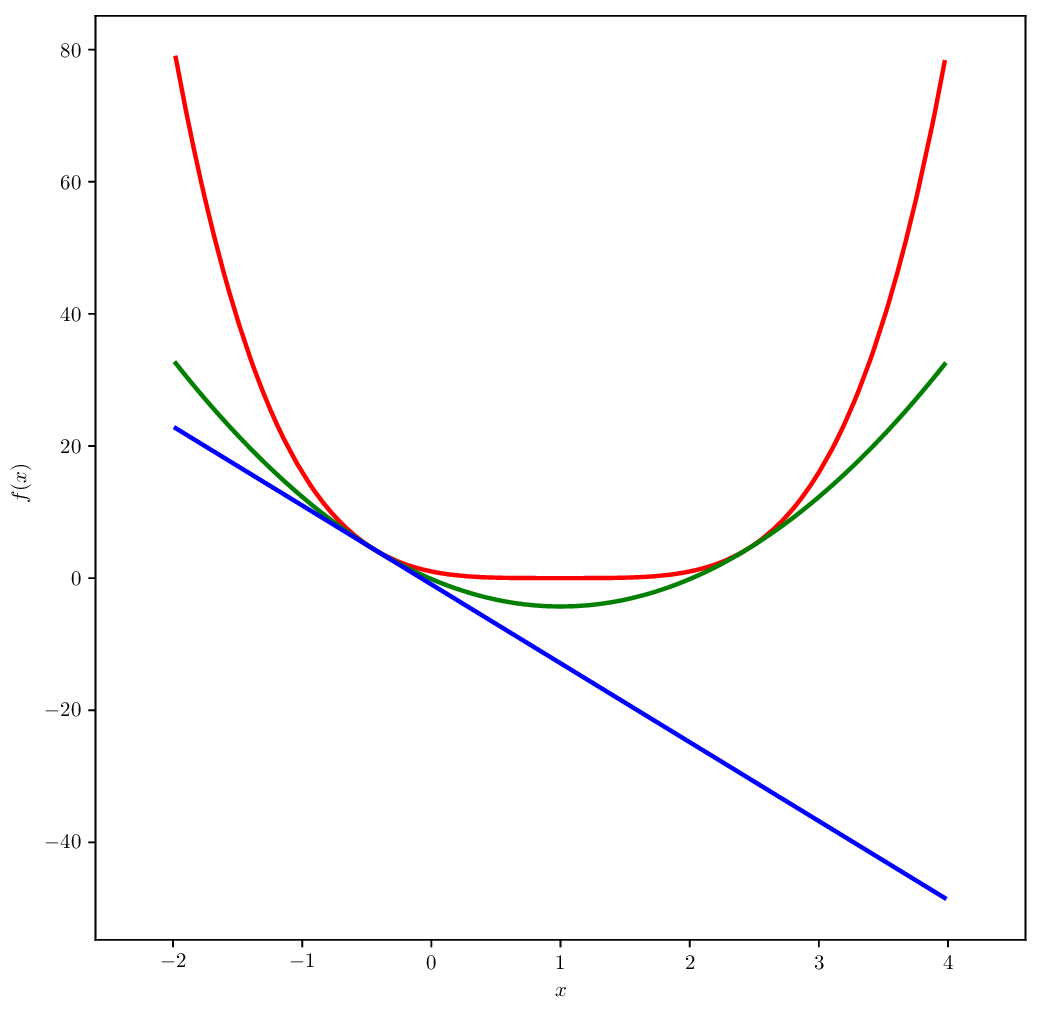}\hfill
	\includegraphics[width=.32\textwidth]{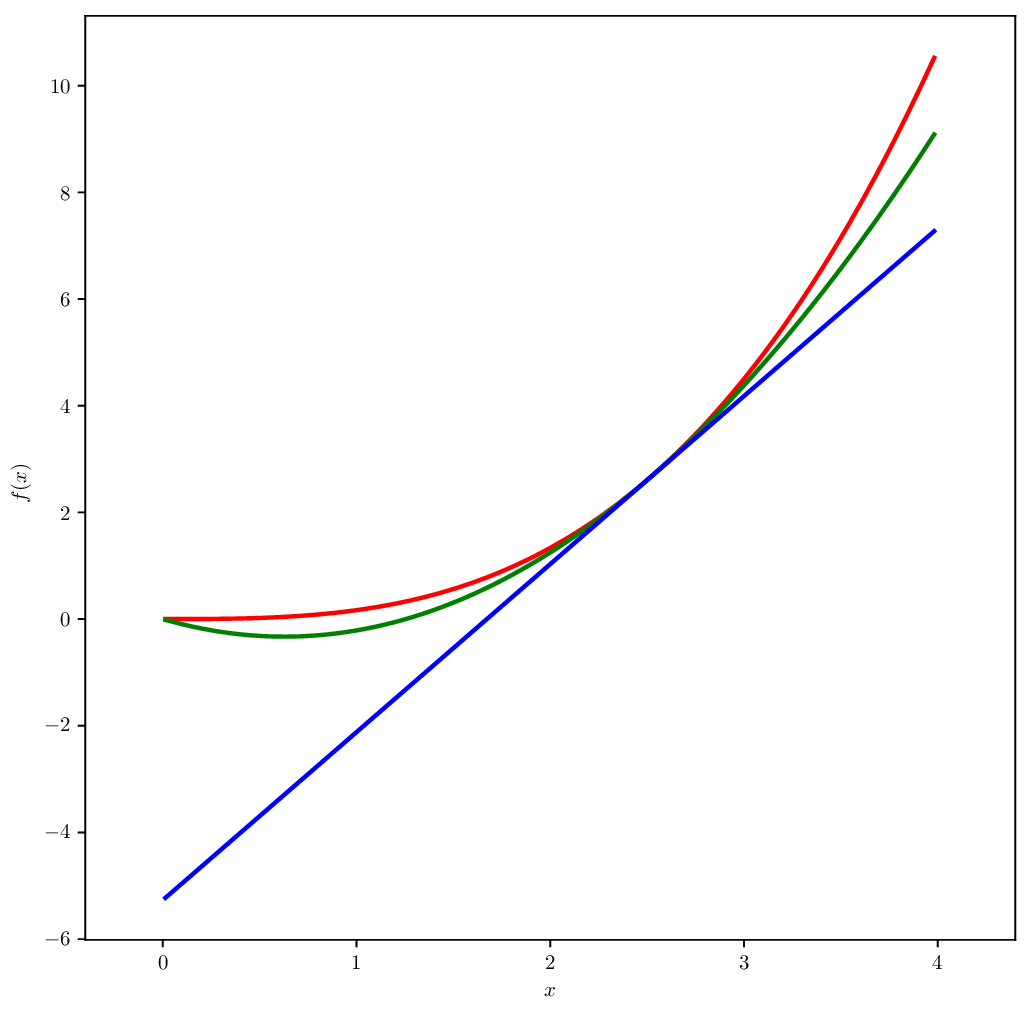}
	\caption{1D quadratic underestimators (green) and linear underestimators (blue) for functions (red) from chakra (left), hs049 (center), and harker (right) benchmark problems.}
	\label{fig:1DQuadUnderestimators}
\end{figure}

\begin{figure}[htb]
	\centering
	\includegraphics[width=.32\textwidth]{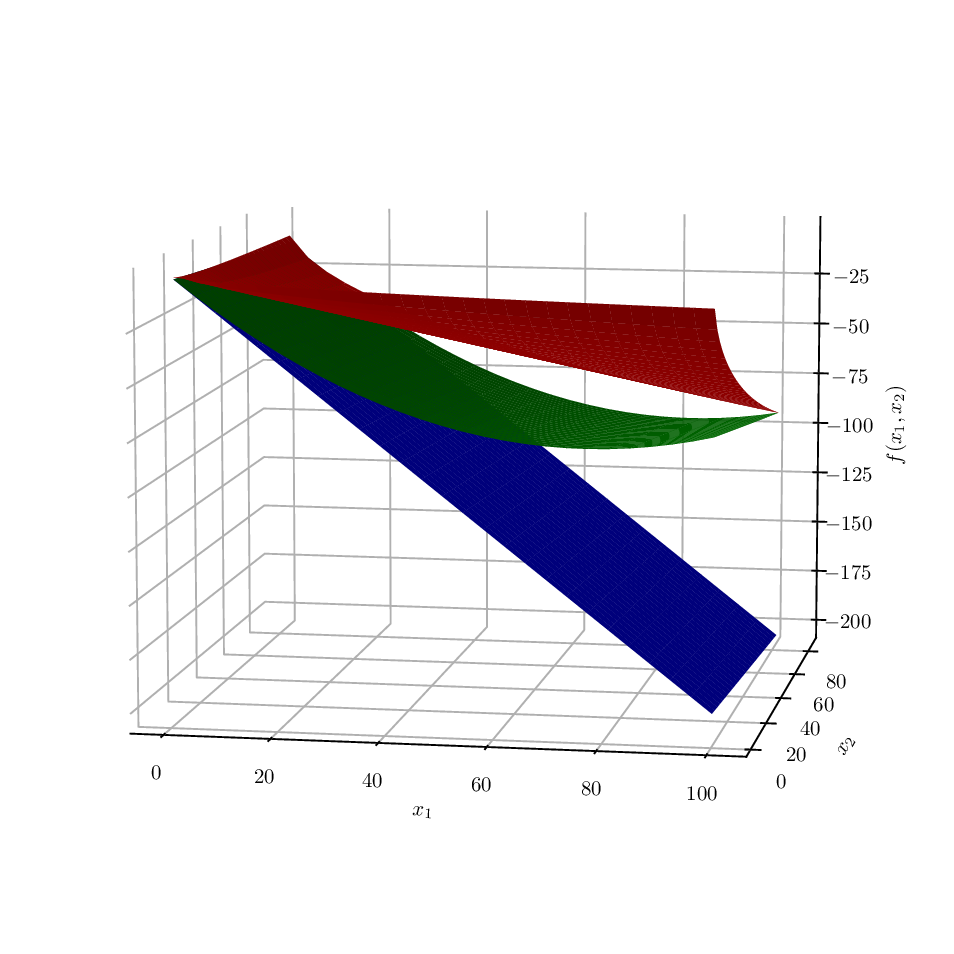}\hfill
	\includegraphics[width=.32\textwidth]{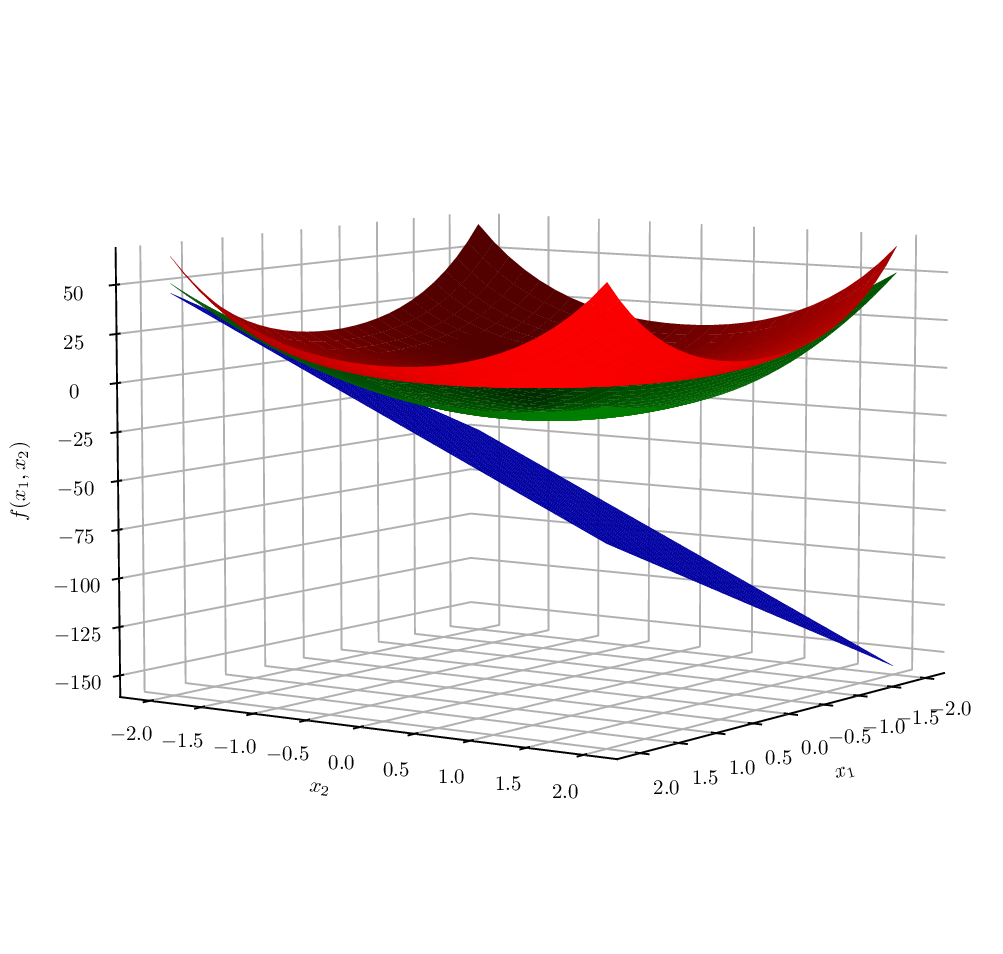}\hfill
	\includegraphics[width=.32\textwidth]{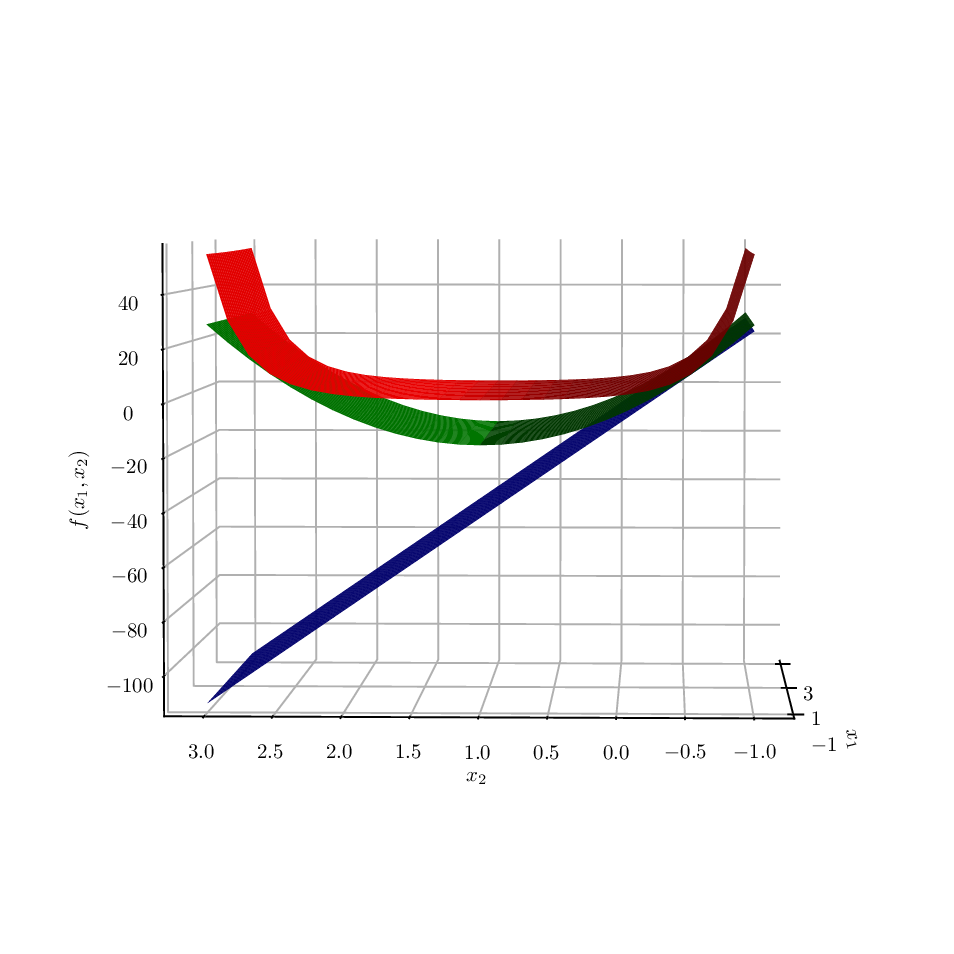}
	\caption{2D quadratic underestimators (green) and linear underestimators (blue) for functions (red) from tls12 (left), arwhead (center), and polak1 (right) benchmark problems.}
	\label{fig:2DQuadUnderestimators}
\end{figure}

\section{Enhanced Underestimators Within Known Polyhedral Regions} \label{sec:externalLinearConstraints}
In the context of DGO algorithms that utilize underestimation, the optimization problem often includes a subset of linear constraints, which define a polyhedral region. We can exploit knowledge of such a region to construct even tighter quadratic underestimators, where we allow them to overestimate functions in points that otherwise violate these linear constraints.

More specifically, given a set of such ``external'' linear constraints $L_j(\x) \le 0$, where $j$ indexes over the latter, the following enumerated list captures the relevant requirements and modifications to the cutting-plane algorithm to accommodate this setting:
\begin{enumerate}
	\item We require that the set of variables of the function that is being underestimated includes all variables referenced in each of the external linear constraints; that is, $\text{dom}(L_j) \subseteq \text{dom}(f)$ for all $j$.
	\item We require that the point of construction selected for the underestimator is within the polyhedral feasible region defined by the linear constraints, which can be enforced through a simple filtering step in preprocessing. This requirement ensures that $\alpha=1$ is a valid upper bound for the initial value of the scaling parameter in the cutting-plane algorithm.
	\item \label{list:mod_list3} We modify Step~\ref{main_algo:init} of the cutting-plane algorithm by executing the vertex enumeration procedure and intersecting the original box domain defined by bounds on the variables with each external linear constraint \textit{before} lifting the domain in $t$. We note that redundant constraints are immediately discarded if $\max\limits_{\x \in \mathcal{V}} L_j(\x) \le 0$, and that Remark~\ref{rem:cut6} applies in this case as well. After attaining a new vertex set resulting from intersecting the box domain with the linear constraints, we then compute bounds on $t$. Whereas $t^U$ remains a maximization of a convex function over a convex set, where evaluation of the vertices suffices, in the case of $t^L$ we include the linear constraints in the convex minimization problem, as shown in~(\ref{eq:min_linconstr}). Here, we remark that including the linear constraints in~(\ref{eq:min_linconstr}) is not necessary for the algorithm to converge, but can create tighter initial bounds on $t$ that may help the algorithm terminate faster.

	\begin{equation}
		\label{eq:min_linconstr}
		\begin{array}{rcl}
			t^L := & \min\limits_{\x \in \mathcal{B}} & f(\x) \\ 
			& \text{s.t.}             & L_j(\x) \le 0 \quad \forall j
		\end{array}
	\end{equation}
\end{enumerate}

Note how, after the above modification, the intersection of the linear constraints with the original box domain generates a set of vertices that are restricted to the feasible domain. Hence, in the course of the cutting-plane algorithm's remaining execution, as the separating hyperplanes are iteratively intersected with this set of vertices that is already restricted to the feasible domain, the cutting-plane algorithm will never check if the quadratic underestimator overestimates the function (Step~\ref{main_algo:evaluate_underestimation}) at a point that is outside of the feasible domain as informed by the external linear constraints. The overall impact is the generation of tighter underestimators, with the accompanying cost of slower convergence as the underestimators will lie closer to the function and the cutting-plane algorithm will require more refinement to elevate the lower bound to the termination tolerance.

We highlight that the added utility of including external linear constraint information in the construction of quadratic underestimators provides a unique advantage to our method compared to those proposed in the literature. We also note that many global optimization algorithms generate cuts to facilitate convergence. In our extension, linear constraints that are not explicit in the problem formulation but are derived through cut generation procedures, whether valid by feasibility or optimality, can be leveraged in this extension to construct tighter quadratic underestimators.

\subsection{Illustrative example}
We provide a simple example to clearly communicate the impact of integrating external linear constraint information into the cutting-plane algorithm to produce tighter underestimators. We consider the function $f(\x) := e^{\frac{1}{2}x_1^2 + x_2^2 + \frac{1}{4}x_1 + \frac{1}{4}x_2 + 1}$ over box $\x \in [0,1]^2$, the point of construction $\x_0 = (1,1)$, and two constraints: $x_1 + x_2 \ge 1$ and $x_1 - x_2 \le 0$.
We qualitatively show the change in the underestimator in Figure~\ref{fig:linconstr}. Quantitatively, the value of $\alpha$ computed without information from external linear constraints is $\alpha=0.3456$, while this value tightens to $\alpha=0.4351$ when $x_1 + x_2 \ge 1$ is added to the problem, and further to $\alpha=0.5261$ when both $x_1 + x_2 \ge 1$ and $x_1 - x_2 \le 0$ are added to the problem. We highlight the clear visual overestimation of the function in the infeasible region presented in Figures~\ref{fig:linconstr} and~\ref{fig:linconstr_topdown}. This simple example demonstrates the capability of Algorithm~\ref{par:algo} to take into account external linear constraint information to generate tighter underestimators in the context of nonlinear optimization problems.

\begin{figure}[!htb]
    \centering
    \begin{subfigure}{0.3\textwidth}
        \centering
        \includegraphics[width=\textwidth]{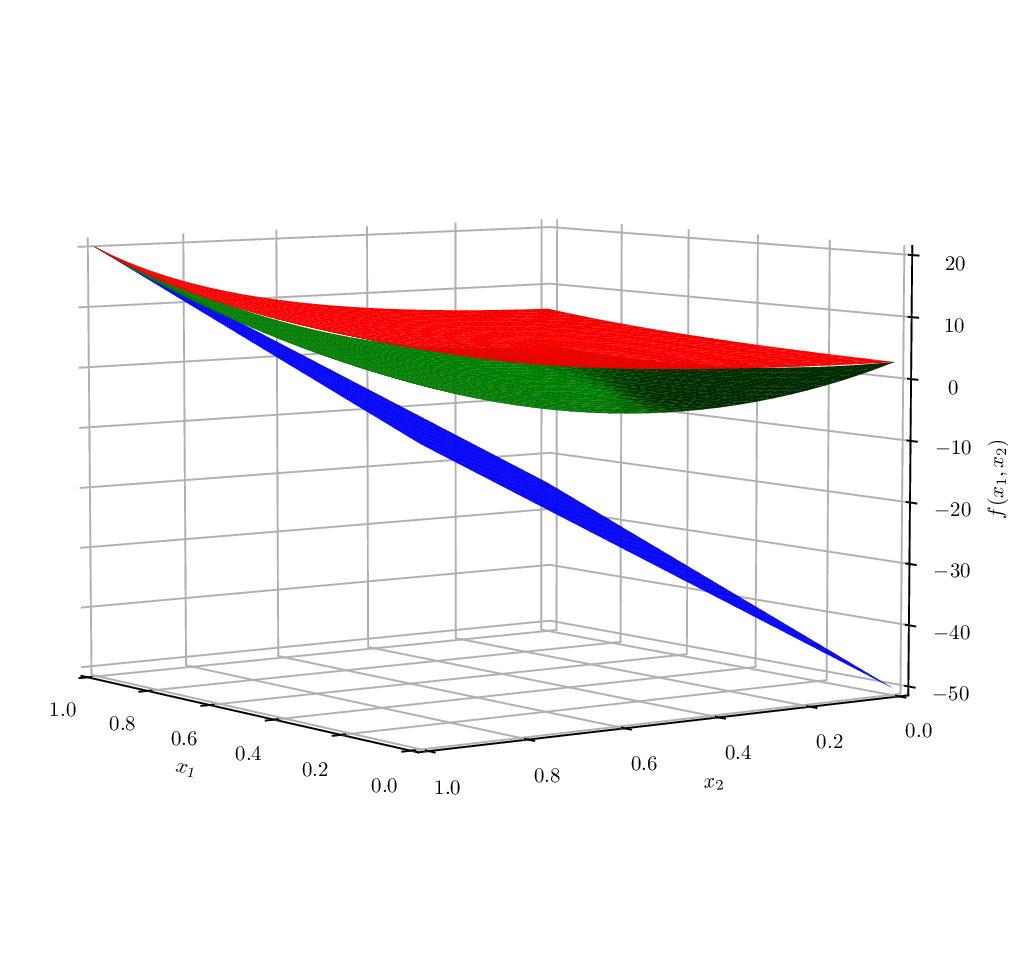}
    \end{subfigure}
    \hfill
    \begin{subfigure}{0.3\textwidth}
        \centering
        \includegraphics[width=\textwidth]{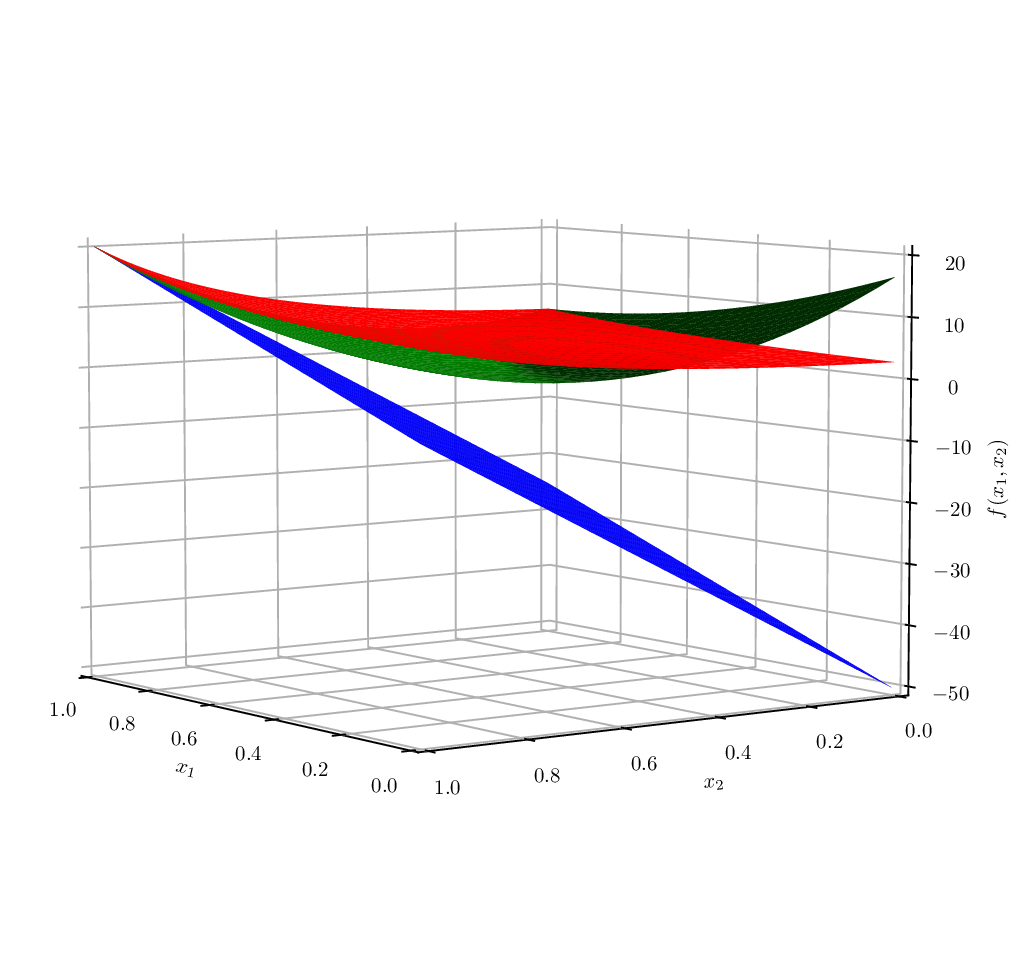}
    \end{subfigure}
    \hfill
    \begin{subfigure}{0.3\textwidth}
        \centering
        \includegraphics[width=\textwidth]{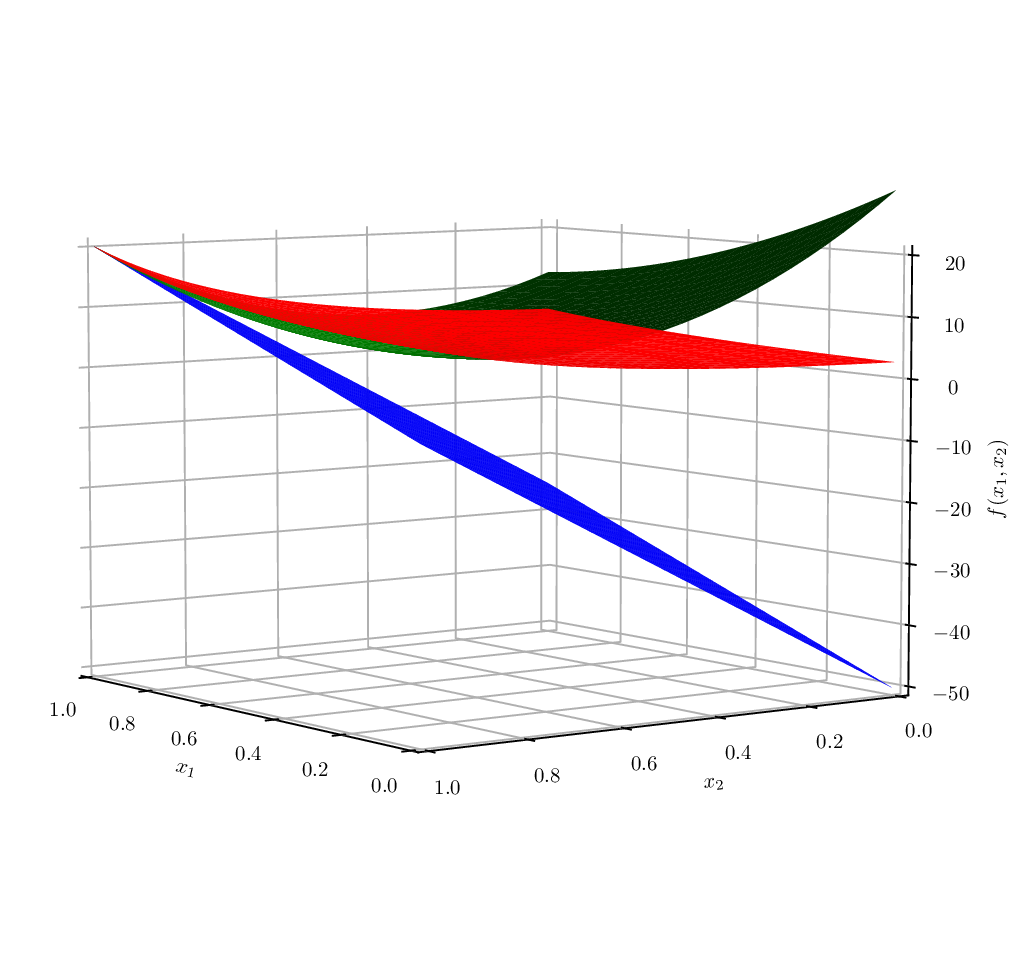}
    \end{subfigure}
    \caption{Quadratic underestimators (green) and linear underestimators (blue) for $e^{\frac{1}{2}x_1^2 + x_2^2 + \frac{1}{4}x_1 + \frac{1}{4}x_2 + 1}$ (red), constructed at point $(1,1)$, in the presence of the following external linear constraints: (i) none (left), (ii) $x_1 + x_2 \ge 1$ (center), (iii) $x_1 + x_2 \ge 1$ and $x_1 - x_2 \le 0$ (right).}
	\label{fig:linconstr}

    \vspace{1em}  

    \begin{subfigure}{0.3\textwidth}
        \centering
        \includegraphics[width=\textwidth]{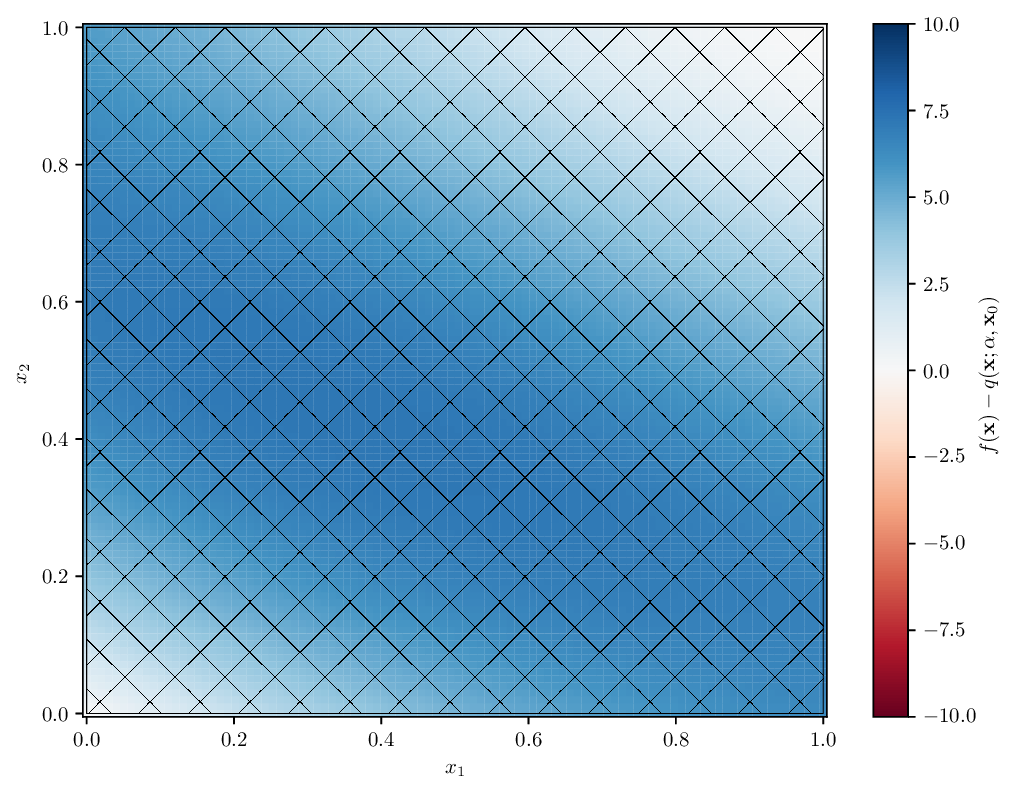}
    \end{subfigure}
    \hfill
    \begin{subfigure}{0.3\textwidth}
        \centering
        \includegraphics[width=\textwidth]{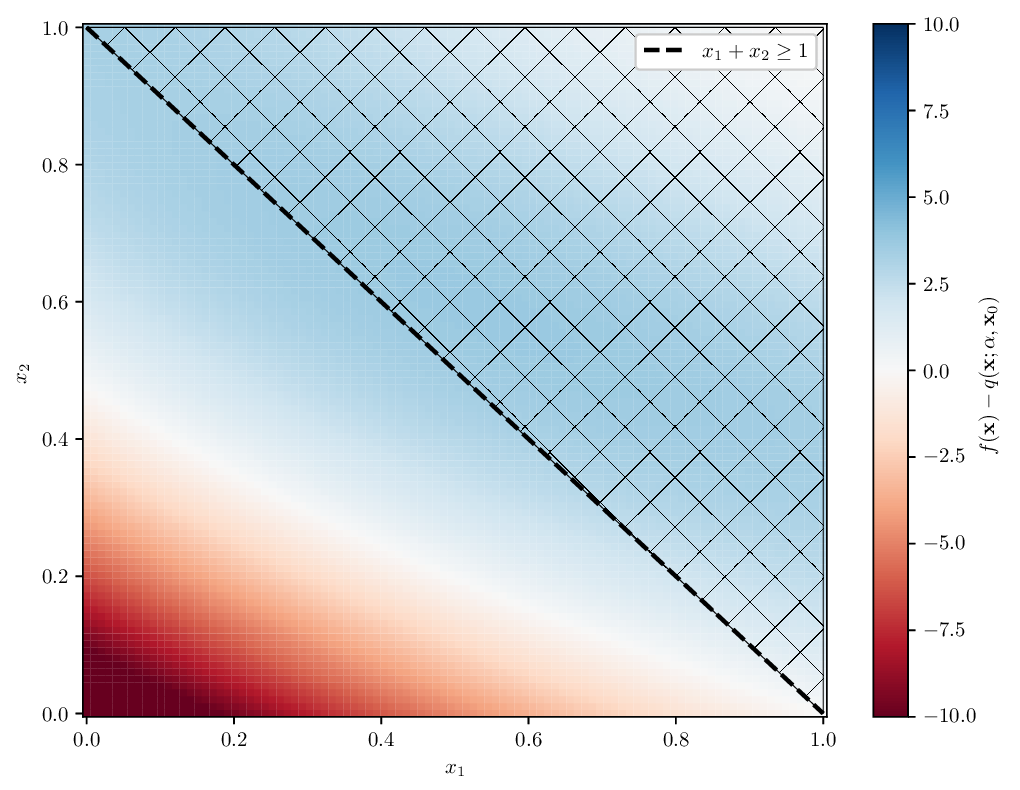}
    \end{subfigure}
    \hfill
    \begin{subfigure}{0.3\textwidth}
        \centering
        \includegraphics[width=\textwidth]{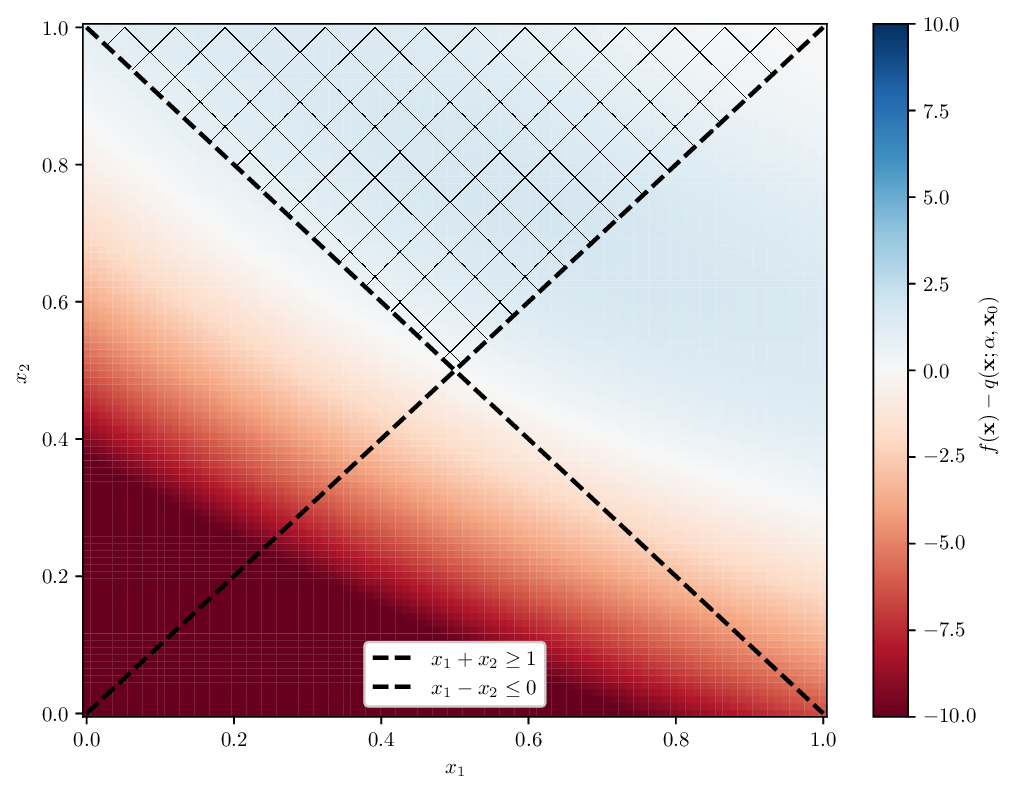}
    \end{subfigure}
    \caption{Difference between the function and the underestimator where coincidence (white), underestimation (blue), overestimation (red), and the feasible region (hatched area) of the problem are displayed. With each added constraint (left : none, center : $x_1 + x_2 \ge 1$, right : $x_1 + x_2 \ge 1$ and $x_1 - x_2 \le 0$), the underestimator becomes tighter as it is allowed to overestimate in the infeasible region.}

	\label{fig:linconstr_topdown}
\end{figure}

\section{Conclusions}\label{sec:conclusions}
	In this work, we presented a cutting-plane algorithm to compute the scaling parameter, $\alpha$, in the second-order Taylor series approximation quadratic underestimator parameter $\alpha$ in the quadratic underestimator proposed by~\citet{su2018improved} for twice-differentiable convex functions.
	We established the convergence of the algorithm and demonstrated the quality of the quadratic underestimators it produces as well as its overall computational efficiency. More specifically, through a computational study involving representative nonlinear convex functions extracted from literature benchmark problems, we showed that quadratic underestimators reduce, on average, the hypervolume between the underestimated function and corresponding linear underestimators constructed at same points of the domain by 53.3\%, 57.5\%, 39.9\%, and 35.4\%, respectively for dimensions 1~through~4. Construction of quadratic underestimators to a convergence tolerance of $\varepsilon=1$e$-3$ was shown to require only a modest amount of computational effort in the order of hundredths or thousandths of a second via a rudimentary implementation. In addition, we proposed an extension of our algorithm to take into account external linear constraint information, which would be available in the context of optimization problems, to further tighten the quadratic underestimators by allowing them to overestimate in the infeasible region, and we demonstrated the impact of this extension with a small example.

	Given the proposed quadratic underestimation procedure to tightly outer approximate twice-differentiable convex functions, there exist opportunities to extend this capability towards outer approximating \textit{non-convex} functions, as well. In particular, we highlight the general class of functions representable as a difference of convex (d.c.) functions for which our cutting-plane algorithm can be readily adapted. We investigate this methodological avenue in the second part of this two-paper series~\citep{strahl2024dc}.

\begin{table}
\caption{1D functions in computational study.}
\label{tab:1DconvexFunctionListMinlplibWithCOCONUT}
 \adjustbox{width=\linewidth}{
 \begin{tabular}{ l l l l l} 
  \toprule
  \multirow{2}{*}{Problem}& \multirow{2}{*}{Objective/Constraint} & \multirow{2}{*}{Expression}& \multirow{2}{*}{Bounds on variables} & \multirow{2}{*}{Parameter Values}\\ [0.5ex]
  & & & & \\
  \midrule
  	procurement2mot$^a$ & e8, 1st term & $f(x_1) = -x_1^{a}$ & \makecell[tl]{$x_1^L$ : 0.10000000000000002 \\ $x_1^U$ : 64.99999999999999 } & a : 0.333333333333333 \\ \hline
	procurement2mot$^a$ & e1, 1st term & $f(x_1) = -x_1^{a}$ & \makecell[tl]{$x_1^L$ : 0.10000000000000002\\ $x_1^U$ : 34.99999999999999 } & a : 0.5 \\ \hline
    chakra$^b$ & objcon, 1st term & $f(x_1) = -ax_1^{b}$ & \makecell[tl]{$x_1^L$ : 1\\ $x_1^{U*}$ : 12 } & \makecell[tl]{a : 0.774245779798168 \\ b : 0.75} \\ \hline
    gtm$^b$ & objcon, 1st term & $f(x_1) = ax_1^{b}$ & \makecell[tl]{$x_1^L$ : 0.2\\ $x_1^{U*}$ : 15 } & \makecell[tl]{a : 8.89583741831423 \\ b : 0.666666666666667} \\ \hline 
    harker$^b$ & objcon, 1st cubic term & $f(x_1) = ax_1^b$ & \makecell[tl]{$x_1^L$ : 0\\ $x_1^{U*}$ : 4 } & \makecell[tl]{a : 0.166666666666667 \\ b : 3}\\ \hline
    hs049$^c$ & objcon, 3rd term & $f(x_1) = -a + x_1^b$ & \makecell[tl]{$x_1^{L*}$ : -2\\ $x_1^{U*}$ : 4} & \makecell[tl]{a : 1 \\ b : 4} \\ \hline 
    schwefel$^d$ & con1, 1st term & $f(x_1) = x_1^{a}$ & \makecell[tl]{$x_1^L$ : -0.5\\ $x_1^U$ : 0.4 } & a : 10 \\ \hline
    srcpm$^b$ & objcon, 1st term & $f(x_1) = ax_1^{-b}$ & \makecell[tl]{$x_1^L$ : 2\\ $x_1^{U*}$ : 25 } & \makecell[tl]{a : 1808.40439881057 \\ b : 2.33333333333333} \\ \hline
	fo7\_ar4\_1$^a$ & e265, 1st term & $f(x_1) = ax_1^{-b}$ & \makecell[tl]{$x_1^L$ : 1.5\\ $x_1^U$ : 6 } & \makecell[tl]{a : 9 \\ b : 1} \\ \hline 
	batchs101006m$^a$ & obj, 11th term & $f(x_1) = a\text{exp}(bx_1)$ & \makecell[tl]{$x_1^L$ : 6.514978663740184\\ $x_1^U$ : 9.615805480084319 } & \makecell[tl]{a : 150 \\ b : 0.5} \\ \hline  
	gams01$^a$ & obj, 4th term & $f(x_1) = a\text{exp}(-bx_1)$ & \makecell[tl]{$x_1^L$ : 0\\ $x_1^U$ : 13.02908889125893 } & \makecell[tl]{a : 985 \\ b : 0.306392275145139} \\ \hline 
    polak3$^c$ & con2, 1st term & $f(x_1) = \text{exp}((-a + x_1)^2)$ &  \makecell[tl]{$x_1^{L*}$ : -1.25\\ $x_1^{U*}$ : 3 }  & a : 0.9092974268256817 \\ \hline
	cvxnonsep\_nsig20r$^a$ & e2, 1st term & $f(x_1) = -a\text{log}(x_1)$ & \makecell[tl]{$x_1^L$ : 1\\ $x_1^U$ : 10 }  & a : 0.065 \\ \hline 
    odfits$^c$ & objcon, 1st term & $f(x_1) = ax_1\text{log}(bx_1)$ & \makecell[tl]{$x_1^L$ : 0.1\\ $x_1^{U*}$ : 10 } & \makecell[tl]{a : 0.5 \\ b : 0.011111111111111112} \\ 
    \bottomrule
    \multicolumn{5}{l}{\small $^a$MINLPLib, $^b$COCONUT/GlobalLib, $^c$COCONUT/CUTE, $^d$COCONUT/CSTP, $^*$bounds assigned} \\
 \end{tabular}}
 \end{table}

\begin{table}
\caption{2D functions in computational study.}
\label{tab:2DconvexFunctionListMinlplibWithCOCONUT}
 \adjustbox{width=\linewidth}{
 \begin{tabular}{ l l l l l} 
  \toprule
  \multirow{2}{*}{Problem}& \multirow{2}{*}{Objective/Constraint} & \multirow{2}{*}{Expression}& \multirow{2}{*}{Bounds on variables} & \multirow{2}{*}{Parameter Values}\\ [0.5ex]
  & & & & \\
  \midrule
	gams01$^a$ & e103, 1st term & $f(x_1, x_2) =  \sqrt{b(-cx_1 - cx_2 + 1)^2 + (dx_1 + dx_2 - 1)^2}$ & \twodbounds{10.500000000000032}{18.45973612666075}{0}{7.959736126660783} & \makecell[tl]{a : 1255.17137031437 \\ b : 0.463731258833818 \\ c : 0.0308163056574352 \\ d : 0.0421228975436196} \\ \hline
	tls12$^a$ & e373, 1st term & $f(x_1, x_2) =  -\sqrt{x_1x_2}$ & \twodbounds{1}{100}{1}{86.0} &  \\ \hline
	cvxnonsep\_pcon20r$^a$ & e2, 1st term & $f(x_1, x_2) =  a^{(x_1 + x_2)}$ & \twodbounds{0}{5}{0}{5}  & a : 2 \\ \hline 
    arwhead$^d$ & obj, 1st nonlinear term & $f(x_1, x_2) = (x_1^2 + x_2^2)^2$ & \twodboundsundef{-2}{2}{-2}{2} & \\ \hline
	batch0812$^a$ & obj, 8th term & $f(x_1, x_2) =  a\text{exp}(x_1 + bx_2)$ & \twodbounds{0.15678251173242852}{ 1.6094379124341}{7.401311768742739}{8.00636756765025}  & \makecell[tl]{a : 1200 \\ b : 0.6}\\ \hline
	batch0812$^a$ & e193, 1st term & $f(x_1, x_2) =  a\text{exp}(-x_1 + x_2)$ & \twodbounds{5.81464756191388}{ 5.9395048081772694}{0.506817602368452}{0.6316748486318415} & a : 485000\\ \hline
    polak1$^c$ & con1, 1st term & $f(x_1, x_2) = \text{exp}(ax_1^2 + (-b + x_2)^2)$ & \twodboundsundef{-1}{3}{-1}{3} & \makecell[tl]{a : 0.001 \\ b : 1}  \\ \hline
	synthes2$^a$ & e1, 1st term & $f(x_1, x_2) =  -\text{log}(x_1 + x_2 + 1)$ & \twodbounds{0}{5.5}{0}{2.75} &  \\
    \bottomrule
    \multicolumn{5}{l}{\small $^a$MINLPLib, $^b$COCONUT/GlobalLib, $^c$COCONUT/CUTE, $^d$COCONUT/CSTP, $^*$bounds assigned} \\
 \end{tabular}}
 \end{table}

\begin{table}
\caption{3D functions in computational study.}
\label{tab:3DconvexFunctionListMinlplibWithCOCONUT}
 \adjustbox{width=\linewidth}{
 \begin{tabular}{ l l l l l} 
  \toprule
  \multirow{2}{*}{Problem}& \multirow{2}{*}{Objective/Constraint} & \multirow{2}{*}{Expression}& \multirow{2}{*}{Bounds on variables} & \multirow{2}{*}{Parameter Values}\\ [0.5ex]
  & & & & \\
  \midrule
	p\_ball\_10b\_5p\_2d\_h$^a$ & e41, 1st term & $f(x_1, x_2, x_3) = (ax_2 + b)\left((\frac{-x_1}{ax_2 + b} + c)^2 + d(\frac{-ex_3}{ax_2 + b} + 1)^2 - 1\right)$ & \threedbounds{0}{6.02957539217313}{0}{1}{0}{10} & \makecell[tl]{a : 0.9999 \\ b : 0.0001 \\ c : 0.648386267690458 \\ d : 28.5367916413211 \\ e : 0.187196372132791}\\ \hline
	clay0203hfsg$^a$ & e106, 1st term & $f(x_1, x_2, x_3) = (ax_2 + b)\left(\frac{cx_1^2}{(x_2 + d)^2} - \frac{ex_1}{ax_2 + b} + \frac{fx_3^2}{(x_2 + d)^2} - \frac{gx_3}{ax_2 + b}\right)$ & \threedbounds{0}{18.5}{0}{1}{0}{13.0} & \makecell[tl]{a : 0.999 \\ b : 0.001 \\ c : 1.00200300400501 \\ d : 0.001001001001001 \\ e : 35 \\ f : 1.00200300400501 \\ g : 14}\\ \hline
	rsyn0805hfsg$^a$ & e376, 1st term & $f(x_1, x_2, x_3) = (ax_2 + b)\left(\frac{x_1}{ax_2 + b} - \text{log}(\frac{x_3}{ax_2 + b} + 1)\right)$ & \threedbounds{0}{2.39789527279837}{0}{1}{0}{10.0}  & \makecell[tl]{a : 0.999 \\ b : 0.001}\\ \hline
	gams01$^a$ & e24, 1st term & $f(x_1, x_2, x_3) = a\sqrt{b(cx_1 + cx_2 - dx_3 - 1)^2 + (ex_1 + ex_2 + fx_3 - 1)^2}$ & \threedbounds{2.500000000000116}{16.42953724218257}{0}{13.929537242182592}{0}{21.260845713230545} & \makecell[tl]{a : 280.391667345843 \\ b : 0.0726707025480232 \\ c : 0.0237333496386274 \\ d : 0.437956808701174 \\ e : 0.0240319698226927 \\ f : 0.175920533216081} \\ \hline
	batchs101006m$^a$ & obj, 1st term & $f(x_1, x_2, x_3) = a\text{exp}(bx_1 + x_2 + x_3)$ & \threedbounds{5.7037824746562}{8.1605182474775}{0}{1.79175946922805}{0}{1.79175946922805} & \makecell[tl]{a : 250 \\ b : 0.6} \\ \hline 
	risk2bpb$^a$ & obj, 1st term & $f(x_1, x_2, x_3) = -a(x_1^{b} + ax_2^{b} + ax_3^{b})$ & \threedbounds{0.5}{ 647499.9311999984}{0.5}{666358.6828799981}{0.5}{580813.0108799983} & \makecell[tl]{a : 0.333333333333333 \\ b : 0.329} \\
    \bottomrule
    \multicolumn{5}{l}{\small $^a$MINLPLib, $^b$COCONUT/GlobalLib, $^c$COCONUT/CUTE, $^d$COCONUT/CSTP, $^*$bounds assigned} \\
 \end{tabular}}
 \end{table}

\begin{table}
\caption{4D functions in computational study.}
\label{tab:4DconvexFunctionListMinlplibWithCOCONUT}
 \adjustbox{width=\linewidth}{
 \begin{tabular}{ l l l l l} 
  \toprule
  \multirow{2}{*}{Problem}& \multirow{2}{*}{Objective/Constraint} & \multirow{2}{*}{Expression}& \multirow{2}{*}{Bounds on variables} & \multirow{2}{*}{Parameter Values}\\ [0.5ex]
  & & & & \\
  \midrule
	p\_ball\_10b\_5p\_3d\_h$^a$ & e62, 1st term & $f(x_1, x_2, x_3, x_4) = (ax_2 + b)\left(c(\frac{-dx_1}{ax_2 + b} + 1)^2 + e(\frac{-fx_3}{ax_2 + b} + 1)^2 + (\frac{-x_4}{ax_2 + b} + g)^2 - 1\right)$ & \makecell[tl]{$x_1^L$ : 0 \\ $x_1^U$ : 10 \\ $x_2^L$ : 0 \\ $x_2^U$ : 1 \\ $x_3^L$ : 0 \\ $x_3^U$ : 10 \\ $x_4^L$ : 0 \\ $x_4^U$ : 10} &  \makecell[tl]{a : 0.9999 \\ b : 0.0001 \\ c : 77.8089905885703 \\ d : 0.113366596747911 \\ e : 90.5954335749978 \\ f : 0.10506228598212 \\ g : 0.894770759747333}\\ \hline
    pspdoc$^c$ & objcon, entire expression & $f(x_1, x_2, x_3, x_4) = \sqrt{x_1^2 + (x_2 - x_3)^2 + 1} + \sqrt{x_2^2 + (x_3 - x_4)^2 + 1}$ & \makecell[tl]{$x_1^{L*}$ : -10 \\ $x_1^U$ : 1 \\ $x_2^{L*}$ : -10 \\ $x_2^{U*}$ : 10 \\ $x_3^{L*}$ : -10 \\ $x_3^{U*}$ : 10 \\ $x_4^{L*}$ : -10 \\ $x_4^{U*}$ : 10} & \\ \hline
    weapon$^d$ & con13, 13th term & $f(x_1, x_2, x_3, x_4) = a\text{exp}(-bx_4-cx_3-dx_2-ex_1)$ & \makecell[tl]{$x_1^{L}$ : 0 \\ $x_1^{U*}$ : 8 \\ $x_2^{L}$ : 0 \\ $x_2^{U*}$ : 8 \\ $x_3^{L}$ : 0 \\ $x_3^{U*}$ : 8 \\ $x_4^{L}$ : 0 \\ $x_4^{U*}$ : 8} & \makecell[tl]{a : 50 \\ b : 0.05129329438755058 \\ c : 0.18632957819149348 \\ d : 0.05129329438755058 \\ e : 0.06187540371808753} \\ \hline
    \bottomrule
    \multicolumn{5}{l}{\small $^a$MINLPLib, $^b$COCONUT/GlobalLib, $^c$COCONUT/CUTE, $^d$COCONUT/CSTP, $^*$bounds assigned} \\
 \end{tabular}}
 \end{table}

\section*{Acknowledgments}
We acknowledge financial support from Mitsubishi Electric Research Labs (MERL) through the Center for Advanced Process Decision-making (CAPD) at Carnegie Mellon University. William Strahl also gratefully acknowledges support from the R.R.~Rothfus Graduate Fellowship in Chemical Engineering and the Chevron Graduate Fellowship in Chemical Engineering.

\bibliographystyle{plain}
\bibliography{StrahlEtal_1_Convex}
\end{document}